\documentclass[12pt]{article}
\usepackage{amsmath,amsfonts,latexsym,mathrsfs,amsthm,amssymb}
\makeatletter
  
  \@addtoreset{equation}{section}
\makeatother
\newtheorem{proposition}{Proposition}[section]
\newtheorem{theorem}{Theorem}[section]
\newtheorem{lemma}{Lemma}[section]
\newtheorem{remark}{Remark}[section]

\newtheorem{example}{Example}[section]
\usepackage{color}
\date{ }

\begin{document}

\title{A system of Schr\"odinger's problems and functional equations
}

\author{Toshio Mikami\thanks{Department of Mathematics, Tsuda University. Partially supported by JSPS KAKENHI Grant Number 24K06765}\and 
Jin Feng\thanks{Department of Mathematics, The University of Kansas}
}

\maketitle

\begin{abstract}
We propose and study a system of Schr\"odinger's problems and functional equations in probability theory.
More precisely, we consider a system of variational problems of relative entropies for probability measures 
 on a Euclidean space  with given two endpoint marginals,
which can be defined inductively.
We also consider an inductively defined system of functional equations, which are Euler's equations for our variational problems.
These are generalizations of Schr\"odinger's problem and functional equation. 
We prove the existence and uniqueness of solutions to our functional equations, 
from which we show the existence and uniqueness of a minimizer of our variational problem. 
Our problem gives an approach for a stochastic optimal transport analog of the Knothe--Rosenblatt rearrangement
 via a variational problem point of view.
 \end{abstract}



\section{Introduction}\label{sec:1}

For a distribution function $F$ on $\mathbb{R}$, a function defined in the following is called the quasi--inverse of $F$:
\begin{equation}\label{1.1}
F^{-1} (u):=\inf\{ x\in \mathbb{R} |u\le F(x)\},\quad 0<u< 1 
\end{equation}
(see e.g., \cite{Nelsen, R1}).
For a metric space $S$, let $\mathcal{P} (S)$ denote the set of all  Borel probability measures on $S$ with weak topology.
For $k=1,\cdots, d$, and $x=(x_j)_{ j=1}^k\in \mathbb{R}^k$, let
\begin{equation}\label{1.2.0}
{\bf x}_i:= (x_j)_{j=1}^i\in \mathbb{R}^i,\quad 1\le i\le k.
\end{equation}
For $d\ge 2, P_i\in \mathcal{P} (\mathbb{R}^d ), i=0,1$ and $k=2,\cdots, d$, 
let $P_i (\cdot |{\bf x}_{k-1})$ denote the regular conditional probability of $P_i $ given ${\bf x}_{k-1}$.
For ${\bf x}_k\in\mathbb{R}^k$, let
\begin{align}\label{1.3.0}
F_{i,k}(x_k|{\bf x}_{k-1}):=&
\begin{cases}  
P_i ((-\infty, x_1]\times \mathbb{R}^{d-1}),&k=1,\\
P_i ((-\infty, x_k]\times \mathbb{R}^{d-k}|{\bf x}_{k-1}), &1<k<d,\\
P_i ((-\infty, x_k]|{\bf x}_{k-1}), &k=d,
\end{cases}\\
T_{k}({\bf x}_k):=&
F_{1,k}(\cdot |(T_1 ({\bf x}_1),\cdots, T_{k-1} ({\bf x}_{k-1})))^{-1}(F_{0,k}(x_k |{\bf x}_{k-1})),\notag\\
T^{KR}_{k}({\bf x}_k):=&(T_{1}({\bf x}_1), \cdots, T_{k}({\bf x}_{k})),\quad 1\le k\le d.\label{1.3.1}
\end{align}
$T^{KR}_d$ is called the Knothe--Rosenblatt rearrangement and plays a crucial role in many fields, e.g., the log--Sobolev inequality, the Brunn--Minkowski inequality, the transportation cost inequality, statistics, and physics (see \cite{Bo1,Bo2, C, KT, K2, L1, R2,T} and the references therein, and also \cite{M2010}).

If $\{F_{0,k}(\cdot |{\bf x}_{k-1})\}_{k=1}^d$ are continuous $P_0$--a.s., then 
$$P_1=P_0(T^{KR}_{d})^{-1}.$$
Let $\delta_x (dy)$ denote  the delta measure on $\{ x\}$ and $p\ge 1$.
$P_0 (dx_1\times \mathbb{R}^{d-1})\delta_{T_{1}(x_1)} (dy_1)$ is a (unique if $p>1$)  minimizer of the following Monge--Kantorovich problem:
\begin{align}\label{1.5}
&\inf\biggl\{\int_{\mathbb{R}\times\mathbb{R}}
|y_1-x_1|^p\pi(dx_1\hbox{ }dy_1):\pi\in \mathcal{P} (\mathbb{R}\times \mathbb{R}),
\\
&\qquad \pi(dx_1\times \mathbb{R})=P_0 (dx_1\times \mathbb{R}^{d-1}),  
\pi(\mathbb{R}\times dy_1)=P_1 (dy_1\times \mathbb{R}^{d-1})\biggr\},\notag
\end{align}
provided it is finite.
For $k=2,\cdots, d$,
$P_0 (d{\bf x}_k\times \mathbb{R}^{d-k})\delta_{T^{KR}_{k}({\bf x}_k)} (d{\bf y}_k)$ is a (unique if $p>1$)  minimizer of the following:
\begin{align} 
&\inf\biggl\{\int_{\mathbb{R}^k\times\mathbb{R}^k}
|y_k-x_k|^p\pi(d{\bf x}_k\hbox{ }d{\bf y}_k):\pi\in \mathcal{P} (\mathbb{R}^k\times \mathbb{R}^k ),\label{1.3}\\
&\qquad\pi(d{\bf x}_k\times \mathbb{R}^{k})=P_0 (d{\bf x}_k\times \mathbb{R}^{d-k}),  
\pi(\mathbb{R}^{k}\times d{\bf y}_k)=P_1 (d{\bf y}_k\times \mathbb{R}^{d-k}),\notag\\
&\qquad \pi(d{\bf x}_{k-1}\times \mathbb{R}\times d{\bf y}_{k-1}\times \mathbb{R})
=P_0 (d{\bf x}_{k-1}\times \mathbb{R}^{d-(k-1)})\delta_{T^{KR}_{k-1}({\bf x}_{k-1})} (d{\bf y}_{k-1})\biggr\},\notag
\end{align}
provided it is finite
(see \cite{D} and also e.g., \cite{R1, V1}), where $d{\bf x}_d\times \mathbb{R}^{0}$ denotes $d{\bf x}_d$.

In \cite{C}, they gave a sequence of minimizers of a class of Monge--Kantorovich problems
that approximates the Knothe--Rosenblatt rearrangement
{\color{black}(see \cite{BPP} for  recent development of the Knothe--Rosenblatt rearrangement).}
Its stochastic optimal transport analog {\color{black}that is called the Knothe--Rosenblatt process}
was discussed in \cite{M2012} (see also \cite{M2010}).
Unlike the Knothe--Rosenblatt rearrangement, no existence theorem of the Knothe--Rosenblatt process exists even though there exist examples.

In this paper, we give an alternative approach for a stochastic optimal transport analog of the Knothe--Rosenblatt rearrangement via 
{\color {black}a system of variational problems of relative entropies for probability measures on a Euclidean space  with given two endpoint marginals.

We describe B. Jamison's results \cite{1-Jamison1974, J75} and explain our problem more precisely.

\begin{theorem}[see \cite{1-Jamison1974}, Theorem 3.2]\label{thm1.1}
Suppose that $S$ is a $\sigma$--compact metric space, that ${\bf m}_1,{\bf m}_2\in \mathcal{P}(S)$, and that $q\in C(S\times S;(0,\infty))$.
Then there exists a unique pair $({\bf m}(dx\hbox{ }dy),{\bf n}_1(dx){\bf n}_2(dy))$ of 
a Borel probability measure and a product $\sigma$--finite measures on $S\times S$ for which 
the following holds:
\begin{align}\label{1.7.0}
{\bf m}(dx\times S)={\bf m}_1(dx), &\quad {\bf m}(S\times dy)={\bf m}_2(dy),\\
{\bf m}(dx\hbox{ }dy)=&q(x,y){\bf n}_1(dx){\bf n}_2(dy).\notag
\end{align}
\end{theorem}

\begin{remark}
(\ref{1.7.0}) is equivalent to the following (see e.g., \cite{M2021} for more discussion):
solve the following equation for ${\bf n}_2(dy)$:
\begin{align*}
{\bf m}_2(dy)=&\left\{\int_S \frac{q(x,y)}{\int_S q(x,\overline y){\bf n}_2(d\overline y)}{\bf m}_1(dx)\right\}{\bf n}_2(dy), 
\end{align*}
and define ${\bf m}(dx\hbox{ }dy)$ by the following:
\begin{align*}
{\bf m}(dx\hbox{ }dy)=&q(x,y)\frac{1}{\int_S q(x,\overline y){\bf n}_2(d\overline y)}{\bf m}_1(dx){\bf n}_2(dy).
\end{align*}
\end{remark}

We describe the assumption and the theorem  in \cite{J75}.

\noindent
(H) $\sigma(t,x)=(\sigma^{ij}(t,x))_{i,j=1}^d$, $(t,x)\in [0,1]\times \mathbb{R}^d$, is a $d\times d$-matrix.
$a(t,x):=\sigma(t,x)\sigma(t,x)^*$, $(t,x)\in [0,1]\times \mathbb{R}^d$, is uniformly nondegenerate,
bounded, once continuously differentiable, and uniformly H\"older continuous, where $\sigma^*$ denotes the transpose of $\sigma$.
$D_x a(t,x)$ is bounded and the first derivatives of $a(t,x)$ are uniformly H\"older continuous
in $x$ uniformly in $t\in [0,1]$.
$b(t,x):[0,1]\times \mathbb{R}^d\longrightarrow \mathbb{R}^d$ is bounded, continuous, and uniformly H\"older continuous in $x$ uniformly in $t\in [0,1]$.

\begin{theorem}[see \cite{J75}]\label{thm1.3}
Suppose that (H) holds.
Then the following stochastic differential equation (SDE for short) has a unique weak solution with a positive continuous transition probability density $p(s,x;t,y)$, $0\le s<t\le 1$, $x,y\in \mathbb{R}^d$:
\begin{eqnarray}\label{1.7}
dX(t)&=&b(t,X(t))dt+\sigma (t,X(t))dB(t),\quad 0< t<1
\end{eqnarray}
where $\{B(t)\}_{0\le t\le 1}$ denotes a Brownian motion.
\end{theorem}

In this paper, we do not fix probability space and we use the same notations $P$ and $B$ for possibly different probabilities and Brownian motions, respectively, when it is not confusing.

Suppose that (H) holds.
For $P_0(dx), P_1(dx)=p_1(x)dx\in \mathcal{P}(\mathbb{R}^d)$,
apply Theorem \ref{thm1.1} for $(S, {\bf m}_1,{\bf m}_2, q(x,y))=(\mathbb{R}^d, P_0, P_1, p(0,x;1,y))$.
Then, from Theorem \ref{thm1.3}, there exists a solution $h(1,\cdot)$ that is unique up to a multiplicative constant to the following Schr\"odinger's functional equation (SFE for short):
\begin{equation}\label{1.4}
P_1(dy)=h(1,y)dy\int_{ \mathbb{R}^d}\frac{p(0,x;1,y)}{
\int_{\mathbb{R}^d}h(1,z)p(0,x;1,z)dz}P_0(dx)
\end{equation}
(see also \cite{B32, Chen, Csiszar,M2021, RT93,S1,S2} and the references therein).

Let
\begin{align}\label{1.10}
h(t,x):=&\int_{\mathbb{R}^d}h(1,z)p(t,x;1,z)dz,\quad (t,x)\in [0,1)\times \mathbb{R}^d.
\end{align}
Then the following is known.
\begin{theorem}[see \cite{J75}]
Suppose that (H) holds.
Then there exists a unique weak solution to the following SDE
that is called the h-path process or Markovian reciprocal process for Brownian motion with initial and terminal distributions $P_0$ and $P_1$, respectively: for $t\in (0,1)$,
\begin{align}
dX_o(t)=&\{a(t, X_o(t))D_x\log  h(t,X_o(t))+b(t, X_o(t))\}dt
+\sigma(t, X_o(t))dB(t),\\
&P(X_o(0),X_o(1))^{-1}(dx\hbox{ }dy)=\frac{h(1,y)}{h(0,x)}P_0(dx)p(0,x;1,y)dy.\label{1.10.1}
\end{align}
In particular, from (\ref{1.4})--(\ref{1.10}),
\begin{equation}
PX_o(t)^{-1}=P_t, \quad t=0,1.\label{1.11.1}
\end{equation}
\end{theorem}
}

We recall the definition of relative entropy:
for $\pi,\tilde\pi\in \mathcal{P}(\mathbb{R}^{d})$, let
\begin{equation}\label{1.14}
H(\tilde\pi\parallel\pi):=
\begin{cases}
\displaystyle\int_{\mathbb{R}^{d}}\left\{\log \frac{d\tilde\pi}{d\pi}(x)\right\}\tilde\pi(dx),&\tilde\pi\ll \pi,\\
+\infty,&{\rm otherwise}.
\end{cases}
\end{equation}
For $m,n\ge 1$ and $(P,Q)\in \mathcal{P}(\mathbb{R}^{m})\times \mathcal{P}(\mathbb{R}^{n})$,
let 
\begin{align}\label{1.11}
\mathcal{A}(P, Q):=&
\{\pi(dx\hbox{ }dy)\in \mathcal{P}(\mathbb{R}^{m}\times \mathbb{R}^{n}):\\
&\qquad \pi(dx\times \mathbb{R}^{n})=P(dx),\pi(\mathbb{R}^{m}\times dy)=Q(dy)\}.\notag
\end{align}
Here, we omit the dependence of $\mathcal{A}(P, Q)$ on $m, n$ except when it is confusing.
Then the following is known:
\begin{align}
&\inf\left\{H(\pi\parallel P_0(dx){\color{black}p(0,x;1,y)}dy):\pi\in \mathcal{A}(P_0, P_1)\right\}\label{1.121}\\
=&H(P(X_o(0),X_o(1))^{-1}\parallel P_0(dx){\color{black}p(0,x;1,y)}dy)\notag\\
=&E\left[ \frac{1}{2}\int_0^1 \left| {\color{black}\sigma(t,X_o(t))}D_x\log  h(t,X_o(t))\right|^2dt\right]\notag\\
=&\inf\left\{E\left[ \frac{1}{2}\int_0^1 |{\color{black}\sigma(t,X(t))^{-1}(b_X(t)-b(t,X(t)))}|^2dt\right]:\right.\label{1.131}\\
&\qquad P(X(0),X(1))^{-1}\in \mathcal{A}(P_0, P_1),\notag\\
&\qquad dX(t)=b_X(t)dt+{\color{black}\sigma(t,X(t))}dB(t), 0<t<1\}.\notag
\end{align}
In (\ref{1.131}), $\{X(t)\}_{0\le t\le 1}$, $\{b_X(t)\}_{0\le t\le 1}$, and $\{B(t)\}_{0\le t\le 1}$ denote a semimartingale,  a progressively measurable stochastic process, and a Brownian motion,
respectively, defined on the same filtered probability space 
(see e.g., \cite{FS06, IW14}).
This variational problem is a class of  Schr\"odinger's problem.
From (\ref{1.11.1}), the following is the minimizer of  (\ref{1.121}):
\begin{equation}\label{1.151}
\frac{h(1,y)}{h(0,x)}P_0(dx){\color{black}p(0,x;1,y)}dy,
\end{equation}
provided (\ref{1.121}) is finite (see \cite{Csiszar,D91, F88,M2021,MT06,RT93,Z86,Z86-2} and references therein and also
(\ref{1.4})--(\ref{1.10}) for notation).
We also call (\ref{1.4}) Schr\"odinger's functional equation for (\ref{1.121}).

{\color{black}
\begin{remark}
If $H(\pi\parallel P_0(dx){\color{black}p(0,x;1,y)}dy)$ is finite and $\pi\in \mathcal{A}(P_0, P_1)$, then
$P_1(dy)\ll dy$.
In particular, ${\bf n}_2 (dy)\ll dy$ in Theorem \ref{thm1.1} .
\end{remark}
}

The Knothe--Rosenblatt process can be defined by generalizing (\ref{1.131}) as a stochastic optimal transport analog of (\ref{1.3}) (see \cite{M2012}).
In this paper, we generalize (\ref{1.121}) as an analog of (\ref{1.3}) and study a new class of functional equations for our variational problem (see (\ref{1.17})--(\ref{1.18}), (\ref{1.19}), and (\ref{1.210})).

We describe notations, and a system of functional equations and variational problems that generalize 
(\ref{1.4}) and (\ref{1.121}).
Let  $d_i\ge 1$ and 
\begin{equation}\label{1.15}
n_i:=\sum_{j=1}^i d_j,\quad  i\ge 1.
\end{equation}
Suppose that there exists $k_0\ge 2$ such that $n_{k_0}=d$.
Let $\{p(x,y)dy\}_{x\in\mathbb{R}^{d}}\subset\mathcal{P}(\mathbb{R}^{d})$.
Suppose that the integral in (\ref{1.16}) below
does not depend on $(x_j)_{j=n_i+1}^d$ (see (A0, i) in section \ref{sec:2} and also Remark  \ref{rk2.1}, (i) in section \ref{sec:2} for a typical example): let
\begin{align}\label{1.16}
&p_{i}({\bf x}_{n_i},{\bf y}_{n_i})\\
:=&
\begin{cases}
\displaystyle
\int_{\mathbb{R}^{d-n_i}}p(x,({\bf y}_{n_i},y))dy,& x=(x_i)_{i=1}^d\in \mathbb{R}^d,
{\bf y}_{n_i}\in \mathbb{R}^{n_i}, 1\le i< k_0,\\
p({\bf x}_{n_{k_0}},{\bf y}_{n_{k_0}}),&{\bf x}_{n_{k_0}},{\bf y}_{n_{k_0}}\in \mathbb{R}^{n_{k_0}}=\mathbb{R}^d, i=k_0.
\end{cases}\notag
\end{align}
For $i=2,\cdots, k_0$, and ${\bf x}_{n_i}, {\bf y}_{n_i}=(y_j)_{j=1}^{n_i}\in \mathbb{R}^{n_i}$, let 
\begin{align}\label{1.170}
{\bf y}_{[n_{i-1}+1,n_i]}:=&(y_j)_{j=n_{i-1}+1}^{n_i},\\
p_{i}({\bf x}_{n_i},{\bf y}_{[n_{i-1}+1,n_i]}|{\bf y}_{n_{i-1}}):=&\frac{p_{i}({\bf x}_{n_i},{\bf y}_{n_i})}
{\int_{\mathbb{R}^{d_i}} p_{i}({\bf x}_{n_i},({\bf y}_{n_{i-1}},z ))dz }. \notag
\end{align}
{\color{black}We use simpler notations such as $x,y$ instead of ${\bf x}_{n_i}, {\bf y}_{[n_{i-1}+1,n_i]}$, etc. when it is not confusing.}
Notice that {\color{black}${\bf y}_{n_i}=({\bf y}_{n_{i-1}}, {\bf y}_{[n_{i-1}+1,n_i]})$
and that $
p_{i}({\bf x}_{n_i},\cdot |{\bf y}_{n_{i-1}})$ is a probability density function on $\mathbb{R}^{d_i}$.} 
For $\mu\in \mathcal{P}(\mathbb{R}^{d})$, let
\begin{equation}\label{1.21.0}
\mu_i(d{\bf x}_{n_i}):=
\begin{cases}
\mu(d{\bf x}_{n_i}\times \mathbb{R}^{d-n_i}),&1\le i< k_0,\\
\mu(d{\bf x}_{n_{k_0}}),&i=k_0.
\end{cases}
\end{equation}
For $\mu,\nu\in \mathcal{P}(\mathbb{R}^{d}),i\ge 2$, and $\tilde \pi\in\mathcal{P}({\mathbb{R}^{n_{i-1}}\times \mathbb{R}^{n_{i-1}}})$, let
\begin{equation}\label{2.8}
\tilde \pi\otimes\mu_{i |  i-1}(d{\bf x}_{n_i}\hbox{ }d{\bf y}_{n_{i-1}})
:=\tilde \pi(d{\bf x}_{n_{i-1}}\hbox{ }d{\bf y}_{n_{i-1}})\mu_i(d{\bf x}_{[n_{i-1}+1,n_i]}|{\bf x}_{n_{i-1}}),
\end{equation}
\begin{align}\label{1.21.1221}
&\mathcal{A}(\mu_i,\nu_i;\tilde \pi)\\
:=&
\{\pi\in\mathcal{A}(\mu_i,\nu_i):
\pi (d{\bf x}_{n_i}\hbox{ }d{\bf y}_{n_{i-1}}\times \mathbb{R}^{d_i})=
\tilde \pi\otimes\mu_{i |  i-1}(d{\bf x}_{n_i}\hbox{ }d{\bf y}_{n_{i-1}})\}\notag
\end{align}
 (see (\ref{1.11}) for notation), and
 \begin{align}
V_i(\mu_i,\nu_i;\tilde \pi)
:=
&\inf\{H(\pi\parallel\tilde \pi\otimes \mu_{i |  i-1}(d{\bf x}_{n_i}\hbox{ }d{\bf y}_{n_{i-1}})p_{i}({\bf x}_{n_i},{\bf y}_{[n_{i-1}+1,n_i]}|{\bf y}_{n_{i-1}})\\
&\qquad \times 
d{\bf y}_{[n_{i-1}+1,n_i]}):
\pi\in \mathcal{A}(\mu_i,\nu_i;\tilde \pi)\}.\notag
\end{align}

The following is our system of variational problems: 
 for $i=1,\cdots, k_0$,
\begin{align}
V_1(\mu_1,\nu_1):=&
\inf\{H(\pi\parallel\mu_1(d{\bf x}_{n_1})p_{1}({\bf x}_{n_1},{\bf y}_{n_1})d{\bf y}_{n_1}):\pi\in \mathcal{A}(\mu_1,\nu_1)\},\quad i=1,\label{1.17}\\
&V_i(\mu_i,\nu_i;\pi_{opt,i-1}), \quad 2\le i\le k_0,\label{1.18}
\end{align}
where $\pi_{opt,i}$ denotes the minimizer of the $i$th problem, provided it exists.
$V_1(\mu_1,\nu_1)$ defined in (\ref{1.17}) is a class of Schr\"odinger's problem. 

Suppose that $\nu$ has a density $f_\nu$ (see (A1) in section \ref{sec:2}).
Then $\nu_i$ has a density $f_{\nu_i}$ defined by the following:
\begin{equation}\label{1.25}
f_{\nu_i}({\bf y}_{n_i}):=\int_{\mathbb{R}^{d-n_i}}f_\nu ({\bf y}_{n_i},z)dz,\quad {\bf y}_{n_i}\in\mathbb{R}^{n_i}, 1\le i< k_0.
\end{equation}
Let $h_1$ be a solution to the following SFE:
\begin{align}\label{1.19}
f_{\nu_1}({\bf y}_{n_1})d{\bf y}_{n_1}
=&h_1({\bf y}_{n_1})d{\bf y}_{n_1}\int_{ \mathbb{R}^{d_1}}
\frac{1}{\int_{\mathbb{R}^{d_1}}h_1(z)p_{1}(x,z)dz}\mu_1(d x)
p_{1}(x,{\bf y}_{n_1})
\end{align}
(see Remark \ref{rk2.1}, (ii) in section \ref{sec:2}).
Then it is known that the measure defined in (\ref{2.3}) below
is the unique minimizer of $V_1(\mu_1,\nu_1)$, i.e., $\pi_{opt,1}$,
provided it is finite (see the references below (\ref{1.151})):
\begin{equation}\label{2.3}
\pi_{\mu_1,\nu_1}(d{\bf x}_{n_1}\hbox{ }d{\bf y}_{n_1}):=
\frac{h_1({\bf y}_{n_1})}{\int_{\mathbb{R}^{d_1}}h_1(z)
p_{1}({\bf x}_{n_1},z)dz}\mu_1(d {\bf x}_{n_1})
p_{1}({\bf x}_{n_1},{\bf y}_{n_1})d{\bf y}_{n_1}.
\end{equation}

For $i=2,\cdots,k_0$, the following is our system of functional equations for the minimizers of (\ref{1.18}):
\begin{align}
f_{\nu_i}({\bf y}_{n_i})d{\bf y}_{n_i}
=&h_i({\bf y}_{n_i})\int_{\mathbb{R}^{n_i}}
\frac{1}{\int_{\mathbb{R}^{d_i}}h_i({\bf y}_{n_{i-1}}, z)
p_{i}(x,z|{\bf y}_{n_{i-1}})dz}\pi_{0,i}(dx\hbox{ }d{\bf y}_{n_i}).\label{1.210}
\end{align}
Here $\pi_{0,i}$ is defined by the following inductively:
\begin{align}\label{2.6}
\pi_{0,i}(d{\bf x}_{n_i}\hbox{ }d{\bf y}_{n_i}):=&\pi_{\mu_{i-1},\nu_{i-1}}\otimes \mu_{i|i-1}(d{\bf x}_{n_i}\hbox{ }d{\bf y}_{n_{i-1}})\\
&\qquad\times p_{i}({\bf x}_{n_i},{\bf y}_{[n_{i-1}+1,n_i]}|{\bf y}_{n_{i-1}})d{\bf y}_{[n_{i-1}+1,n_i]},\notag\\
\pi_{\mu_i,\nu_i}(d{\bf x}_{n_i}\hbox{ }d{\bf y}_{n_i})
:=&\frac{h_i({\bf y}_{n_i})}{\int_{\mathbb{R}^{d_i}}h_i({\bf y}_{n_{i-1}}, z)
p_{i}({\bf x}_{n_i},z|{\bf y}_{n_{i-1}})dz}\pi_{0,i}(d{\bf x}_{n_i}\hbox{ }d{\bf y}_{n_i})\label{1.27}
\end{align}
(see (\ref{2.8}) for notation and (\ref{2.10.1}) in section \ref{sec:2}).

\begin{remark}\label{rk2.2}
For $i=2,\cdots,k_0, Q\in \mathcal{P}(\mathbb{R}^{n_i})$ and 
$\pi\in \mathcal{A}(\mu_i,Q;\pi_{\mu_{i-1},\nu_{i-1}})$, 
$$\pi (d{\bf x}_{n_{i-1}}\times \mathbb{R}^{d_i}\times d{\bf y}_{n_{i-1}}\times \mathbb{R}^{d_i})=\pi_{\mu_{i-1},\nu_{i-1}}(d{\bf x}_{n_{i-1}}\hbox{ }d{\bf y}_{n_{i-1}}).$$
$\pi_{0,i}\in \cup_{P\in \mathcal{P}(\mathbb{R}^{n_i})}\mathcal{A}(\mu_i,P;\pi_{\mu_{i-1},\nu_{i-1}})$,
provided (\ref{1.210}) has a solution.
\end{remark}

Even when variational problems (\ref{1.18}) are infinite, one can consider functional equations (\ref{1.210})
as there exists a (unique up to a multiplicative constant) solution to the SFE  (\ref{1.19}) for $V_1(\mu_1,\nu_1)$ even when it is infinite (see \cite{1-Jamison1974}).
In section \ref{sec:2}, we show that (\ref{1.210}) has a solution and that $\pi_{\mu_i,\nu_i}=\pi_{opt,i}$, $2\le i\le k_0$, provided (\ref{1.18}) is finite. 

{\color{black}
The zero--noise limit of (\ref{1.17}) solves Monge's problem (see \cite{M2004} and also \cite{
Leo1, Leo2, M2021} and the references therein).
R. Fortet \cite {1-Fo} solved the SFE by a successive approximation, which is called the Sinkhorn algorithm in data science nowadays (see \cite{CGT,1-PC} and the references therein).
The studies of the zero--noise limit of  (\ref{1.18}) and of an algorithm for functional equations (\ref{1.210}) are our future problem.
The duality theory for (\ref{1.18}) should be also studied.
A nice property of the Knothe--Rosenblatt rearrangement is its explicit formula (\ref{1.3.1}).
On the other hand, the role of the Knothe--Rosenblatt rearrangement in optimal transport has not been studied deeply.
We hope our result provides some insight into this  in the future.}

In section \ref{sec:2}, we state our result.
In section \ref{sec:3}, we give technical lemmas and prove our results in section \ref{sec:4}.
{\color{black}In the Appendix, we give the proofs of Example \ref{ex2.1} and Lemma \ref{lm3.1}.}

\section{Main result}\label{sec:2}

In this section, we state our results.

We describe assumptions and notations to state our results.
As of this section, $\mu,\nu\in \mathcal{P}(\mathbb{R}^{d})$.

\noindent
(A0)
(i) $p\in C(\mathbb{R}^{d}\times \mathbb{R}^{d};(0,\infty))$ and $\{p(x,y)dy\}_{x\in\mathbb{R}^{d}}\subset\mathcal{P}(\mathbb{R}^{d})$. 
$p_{i}({\bf x}_{n_i},{\bf y}_{n_i})$ does not depend on $(x_j)_{j=n_i+1}^d$ and is continuous in $({\bf x}_{n_i},{\bf y}_{n_i})\in \mathbb{R}^{n_i}\times \mathbb{R}^{n_i}$, $1\le i<k_0$ (see (\ref{1.2.0}) and (\ref{1.15})--(\ref{1.16}) for notations).

\noindent
(ii) There exists a function  $\psi_1\in C(\mathbb{R}^{d_1})$ such that for ${\bf x}_{n_1}\in \mathbb{R}^{d_1}$, 
the following is convex:
$$\mathbb{R}^{d_1}\ni  y\mapsto \log p_{1}({\bf x}_{n_1}, y)+\psi_1( y).$$

\noindent
(iii) For $i=2,\cdots,k_0$, there exists a  function $\psi_i\in C(\mathbb{R}^{d_i})$ such that for $({\bf x}_{n_i},{\bf y}_{n_{i-1}})\in \mathbb{R}^{n_i}\times \mathbb{R}^{n_{i-1}}$, 
the following is convex:
$$\mathbb{R}^{d_i}\ni y\mapsto \log p_{i}({\bf x}_{n_i},({\bf y}_{n_{i-1}},y))+\psi_i(y).$$

\noindent
(A1)
(i) $\nu$ has a probability density $f_\nu$.

\noindent
(ii) $f_{\nu_i}\in C(\mathbb{R}^{n_i})$, $1\le i\le k_0$ (see (\ref{1.21.0}) and (\ref{1.25})  for notation).

\begin{remark}\label{rk2.1}
{\color{black}(i) In (\ref{1.7}), suppose that 
$\sigma\in C^\infty_b (\mathbb{R}^{d};M(d, \mathbb{R}))$ and is uniformly nondegenerate, and $b\in C^\infty_b(\mathbb{R}^{d};\mathbb{R}^{d})$.
Then there exists $C>0$ such that for  any $x\in\mathbb{R}^{d}$, $y\mapsto \log p(0,x;1,y)+C|y|^2$ is convex (see \cite{Sheu91}).
In particular, $p(0,x;1,y)$ satisfies (A0, ii, iii).
Suppose, in addition, that $a(x)=(a_{ij}({\bf x}_{\max(i,j)}))_{i,j=1}^d$ and  
$b(x)=(b_i({\bf x}_{i}))_{i=1}^d$.
Then (A0, i) holds.}\\
(ii) 
(A0, i) implies that $p_{1}\in C(\mathbb{R}^{d_1}\times \mathbb{R}^{d_1};(0,\infty))$ and  the SFE (\ref{1.19})
has a solution $h_1$ that is unique up to a multiplicative constant  (see Theorem \ref{thm1.1} in section \ref{sec:1}).
In particular, $\pi_{\mu_1,\nu_1}\otimes \mu_{2|1}(d{\bf x}_{n_2}\hbox{ }d{\bf y}_{n_1})$ can be defined.\\
(iii) Under (A1, ii), $f_{\nu_i}^{-1}((0,\infty)), 1\le i\le k_0$ are open sets, which plays a crucial role in the proof of our result.\\
\end{remark}

We describe notations, provided $h_i:\mathbb{R}^{n_i}\rightarrow [0,\infty)$ {\color{black}exists and  
$h_i({\bf y}_{n_{i-1}},\cdot)$ is 
 measurable for ${\bf y}_{n_{i-1}}\in \mathbb{R}^{n_{i-1}}, 1\le i\le k_0$, where $({\bf y}_{n_0},z)$ denotes $z$ for $z\in \mathbb{R}^{n_1}$}
 (see (\ref{1.19}) and (\ref{1.210})--(\ref{1.27}) for notation).
For ${\bf x}_{n_i}, {\bf y}_{n_i}\in \mathbb{R}^{n_i}$, let
\begin{align}
h_i (0, {\bf x}_{n_i}, {\bf y}_{n_{i-1}})
:=&
\begin{cases}
\displaystyle h_1(0,{\bf x}_{n_1}):=
\int_{\mathbb{R}^{d_1}}h_1(z)p_1({\bf x}_{n_1},z)dz,&i=1,\\
\displaystyle \int_{\mathbb{R}^{d_i}}h_i({\bf y}_{n_{i-1}}, z)
p_{i}({\bf x}_{n_i},({\bf y}_{n_{i-1}},z))dz,&2\le i\le k_0,
\end{cases}\label{3.12}\\
\pi_{\mu_{i},\nu_{i}}({\bf y}_{n_i}, d{\bf x}_{n_i})
:=&\frac{1}{f_{\nu_{i}}( {\bf y}_{n_i})}
\prod_{j=1}^{i}\frac{h_j ({\bf y}_{n_j})p_j({\bf x}_{n_j},{\bf y}_{n_j})}{h_j (0, {\bf x}_{n_j}, {\bf y}_{n_{j-1}})}\mu_{i}(d{\bf x}_{n_i}),\label{2.5}
\end{align}
provided $f_{\nu_{i}}( {\bf y}_{n_i})>0, 1\le i\le k_0$, and 
\begin{align}
&\pi_{\mu_{i-1},\nu_{i-1}}\otimes\mu_{i|i-1}({\bf y}_{n_{i-1}},d{\bf x}_{n_i})\label{2.3.1223}\\
:=&\pi_{\mu_{i-1},\nu_{i-1}}({\bf y}_{n_{i-1}}, d{\bf x}_{n_{i-1}})\mu_{i}(d{\bf x}_{[n_{i-1}+1,n_i]}|{\bf x}_{n_{i-1}}),\notag
\end{align}
provided $f_{\nu_{i-1}}( {\bf y}_{n_{i-1}})>0, 2\le i\le k_0$.
For $i=2,\cdots,k_0$,  $Q_1\in \mathcal{P}(\mathbb{R}^{n_i})$, $Q_2\in \mathcal{P}(\mathbb{R}^{d_i})$, and 
${\bf y}_{n_{i-1}}\in \mathbb{R}^{n_{i-1}}$ such that $f_{\nu_{i-1}}({\bf y}_{n_{i-1}})>0$,
let
\begin{align}\label{2.30}
&V_i(Q_1,Q_2;{\bf y}_{n_{i-1}})\\
:=&\inf\{H(\pi \parallel\pi_{0,i} (d{\bf x}_{n_i}\hbox{ }d{\bf y}_{[n_{i-1}+1,n_i]}|{\bf y}_{n_{i-1}})):\pi\in\mathcal{A}(Q_1,Q_2)\},\notag
\end{align}
\begin{equation}
f_{\nu_i}(y|{\bf y}_{n_{i-1}}):=\frac{f_{\nu_i}({\bf y}_{n_{i-1}},y)}{f_{\nu_{i-1}}({\bf y}_{n_{i-1}})}, \quad  y\in\mathbb{R}^{d_i}.\label{2.5.1}
\end{equation}
\begin{remark}\label{rk2.2.0}
Suppose that $h_i:\mathbb{R}^{n_i}\rightarrow [0,\infty), 1\le i\le k_0$ exist and are 
measurable.
Then, from (\ref{2.8}) and (\ref{2.6}), for $i=2,\cdots, k_0$,
\begin{align}
&\pi_{\mu_{i-1},\nu_{i-1}}\otimes\mu_{i|i-1}(d{\bf x}_{n_i}|{\bf y}_{n_{i-1}})\label{2.62}\\
=&
\pi_{\mu_{i-1},\nu_{i-1}}(d{\bf x}_{n_{i-1}}|{\bf y}_{n_{i-1}})\mu_{i}(d{\bf x}_{[n_{i-1}+1,n_i]}|{\bf x}_{n_{i-1}}),\notag\\
&\pi_{0,i} (d{\bf x}_{n_i}\hbox{ }d{\bf y}_{[n_{i-1}+1,n_i]}|{\bf y}_{n_{i-1}})\label{2.3.0}\\
=&
\pi_{\mu_{i-1},\nu_{i-1}}\otimes\mu_{i|i-1}(d{\bf x}_{n_i}|{\bf y}_{n_{i-1}})  p_{i}({\bf x}_{n_i},{\bf y}_{[n_{i-1}+1,n_i]}|{\bf y}_{n_{i-1}})d{\bf y}_{[n_{i-1}+1,n_i]}.\notag
\end{align}
Since $\pi_{\mu_{i},\nu_{i}}\in \mathcal{A}(\mu_{i},\nu_{i})$, the following holds $f_{\nu_{i}}( {\bf y}_{n_i})d{\bf y}_{n_i}$--a.e. (see (\ref{1.210}), (\ref{1.27}), and (\ref{2.10.1})):
\begin{align*}
\pi_{\mu_{i},\nu_{i}}(d{\bf x}_{n_i}|{\bf y}_{n_i})
=&\pi_{\mu_{i},\nu_{i}}({\bf y}_{n_i}, d{\bf x}_{n_i}), \quad 1\le i\le k_0.
\end{align*}
In particular, for $i=2,\cdots, k_0$, the following also holds $f_{\nu_{i-1}}( {\bf y}_{n_{i-1}})d{\bf y}_{n_{i-1}}$--a.e.:
\begin{align}\label{2.6.0}
\pi_{\mu_{i-1},\nu_{i-1}}\otimes\mu_{i|i-1}(d{\bf x}_{n_i}|{\bf y}_{n_{i-1}})
=&\pi_{\mu_{i-1},\nu_{i-1}}\otimes\mu_{i|i-1}({\bf y}_{n_{i-1}},d{\bf x}_{n_i}).
\end{align}
\end{remark}

The following plays a crucial role in the proof of the main result.

\begin{proposition}\label{pp2.1}
Suppose that  (A0, i, iii) and (A1, i) hold.
Then, for $i=2,\cdots, k_0$ and ${\bf y}_{n_{i-1}}\in \mathbb{R}^{n_{i-1}}$ such that $f_{\nu_{i-1}}({\bf y}_{n_{i-1}})>0$,
the following has a solution that is unique up to a multiplicative function of ${\bf y}_{n_{i-1}}$:
\begin{align}\label{2.10}
&f_{\nu_i}({\bf y}_{[n_{i-1}+1,n_i]}|{\bf y}_{n_{i-1}})d{\bf y}_{[n_{i-1}+1,n_i]}\\
=&h_i({\bf y}_{n_{i-1}},{\bf y}_{[n_{i-1}+1,n_i]})d{\bf y}_{[n_{i-1}+1,n_i]}\notag\\
&\qquad\times \int_{\mathbb{R}^{n_i}}
\frac{p_i(x,{\bf y}_{[n_{i-1}+1,n_i]}|{\bf y}_{n_{i-1}})
}{\int_{\mathbb{R}^{d_i}}h_i({\bf y}_{n_{i-1}}, y)p_i(x,y|{\bf y}_{n_{i-1}})dy}
\pi_{\mu_{i-1},\nu_{i-1}}\otimes\mu_{i|i-1}({\bf y}_{n_{i-1}},dx),\notag
\end{align}
provided $\pi_{\mu_{i-1},\nu_{i-1}}({\bf y}_{n_{i-1}}, d{\bf x}_{n_{i-1}})\in \mathcal{P}(\mathbb{R}^{n_{i-1}})$.
In particular, the measure defined by
\begin{align}\label{2.7.1}
&\pi_{\mu_i, \nu_i }({\bf y}_{n_{i-1}},d{\bf x}_{n_i}\hbox{ }d{\bf y}_{[n_{i-1}+1,n_i]})\\
:=&
\frac{h_i ({\bf y}_{n_i})p_i({\bf x}_{n_i},{\bf y}_{n_i})}{h_i (0, {\bf x}_{n_i}, {\bf y}_{n_{i-1}})}
\pi_{\mu_{i-1},\nu_{i-1}}\otimes\mu_{i|i-1}({\bf y}_{n_{i-1}},d{\bf x}_{n_i})d{\bf y}_{[n_{i-1}+1,n_i]}\notag
\end{align}
belongs to $\mathcal{A}(\pi_{\mu_{i-1},\nu_{i-1}}\otimes\mu_{i|i-1}({\bf y}_{n_{i-1}},d{\bf x}_{n_i}),
f_{\nu_i}({\bf y}_{[n_{i-1}+1,n_i]}|{\bf y}_{n_{i-1}})d{\bf y}_{[n_{i-1}+1,n_i]})$.
It is also the unique minimizer of the following,
$f_{\nu_{i-1}}({\bf y}_{n_{i-1}})d{\bf y}_{n_{i-1}}$--a.e.:
\begin{equation}\label{2.11}
V_i(\pi_{\mu_{i-1},\nu_{i-1}}\otimes\mu_{i|i-1}({\bf y}_{n_{i-1}},d{\bf x}_{n_i}),
f_{\nu_i}({\bf y}_{[n_{i-1}+1,n_i]}|{\bf y}_{n_{i-1}})d{\bf y}_{[n_{i-1}+1,n_i]};{\bf y}_{n_{i-1}}),%
\end{equation}
provided that it is finite,  that $\pi_{\mu_{i-1},\nu_{i-1}}({\bf y}_{n_{i-1}}, d{\bf x}_{n_{i-1}})\in \mathcal{P}(\mathbb{R}^{n_{i-1}})$, and  that 
(\ref{2.3.0}) and (\ref{2.6.0}) hold.
\end{proposition}
\begin{proposition}\label{pp2.4}
Suppose that (A0)--(A1) hold.
Then there exists a continuous solution $h_1$ of (\ref{1.19}) such that
$\pi_{\mu_{1},\nu_{1}}({\bf y}_{n_1}, d{\bf x}_{n_1})\in \mathcal{P}(\mathbb{R}^{n_{1}})$ for ${\bf y}_{n_1}\in \mathbb{R}^{d_1}$ for which $ f_{\nu_{1}}({\bf y}_{n_1})>0$,
and that 
$f_{\nu_1}^{-1}((0,\infty))\ni{\bf y}_{n_1}\mapsto \pi_{\mu_{2}, \nu_{2}}({\bf y}_{n_1},d{\bf x}_{n_2}
 d{\bf y}_{[n_1+1,n_2]})$
is weakly continuous, i.e., 
 for any $\varphi \in C_0(\mathbb{R}^{n_2}\times \mathbb{R}^{d_2})$
and ${\bf y}_{n_1}\in \mathbb{R}^{n_1}$ for which $ f_{\nu_{1}}({\bf y}_{n_1})>0$,
\begin{align}\label{2.12weakly continuous}
&\lim_{{\bf z}_{n_1}\to {\bf y}_{n_1}}\int_{\mathbb{R}^{n_2}\times \mathbb{R}^{d_2}}
\varphi (x,y)\pi_{\mu_{2}, \nu_{2}}({\bf z}_{n_1},dx \hbox{ }dy)\\
=&\int_{\mathbb{R}^{n_2}\times \mathbb{R}^{d_2}}
\varphi (x,y)\pi_{\mu_{2}, \nu_{2}}({\bf y}_{n_1},dx \hbox{ }dy).\notag
\end{align}
\end{proposition}

In (\ref{1.3.0})--(\ref{1.3.1}), the Knothe--Rosenblatt rearrangements $T^{KR}_i, 1\le i\le d$ are defined,
and give the minimizers of variational problems (\ref{1.5})--(\ref{1.3}), provided they are finite.
In the following, instead of defining mappings, we consider a system of functional equations (\ref{1.210})
from which we describe the minimizers of a system of variational problems (\ref{1.18}) that can be considered an analog of (\ref{1.3}) (see \cite{M2004} and also \cite{Leo1,Leo2,M2021} for the relation between $T_1$ and (\ref{1.5}), and (\ref{1.17})).
(\ref{1.18}) and (\ref{1.210}) can be also considered generalizations of Schr\"odinger's problem and 
Schr\"odinger's functional equation, respectively.
We recall that $\pi_{\mu_1,\nu_1}=\pi_{opt,1}$, provided $V_1(\mu_1,\nu_1)$ is finite (see (\ref{2.3})).

\begin{theorem}\label{thm2.1}
Suppose that (A0)--(A1) hold.
Then  for $i=2,\cdots, k_0$, 
there exists a 
measurable function $h_i$ that satisfies  (\ref{2.10}), $f_{\nu_{i-1}}( {\bf y}_{n_{i-1}})d{\bf y}_{n_{i-1}}$--a.e. and 
such that  $h_i({\bf y}_{n_{i-1}},\cdot)\in C(\mathbb{R}^{d_i})$ for ${\bf y}_{n_{i-1}}\in \mathbb{R}^{n_{i-1}}$.
In particular, $h_i$ is  a solution to (\ref{1.210}) that is unique up to a multiplicative 
measurable function of ${\bf y}_{n_{i-1}}$, and
\begin{align}\label{2.10.1}
\pi_{\mu_i,\nu_i}(d{\bf x}_{n_i}\hbox{ }d{\bf y}_{n_i})
=&\prod_{j=1}^i\frac{h_j ({\bf y}_{n_j})p_j({\bf x}_{n_j},{\bf y}_{n_j})}
{h_j(0, {\bf x}_{n_j}, {\bf y}_{n_{j-1}})}\mu_i(d{\bf x}_{n_i})d{\bf y}_{n_i}\\
\in &\mathcal{A}(\mu_i,\nu_i;\pi_{\mu_{i-1},\nu_{i-1}}).\notag
\end{align}
$\pi_{\mu_i,\nu_i}$ is the unique minimizer of  $V_i(\mu_i,\nu_i;\pi_{\mu_{i-1},\nu_{i-1}})$, provided it is finite.
\end{theorem}

\begin{remark}
In (\ref{1.210}), $h_i({\bf y}_{n_i})=h_i({\bf y}_{n_{i-1}},{\bf y}_{[n_{i-1}+1,n_i]})$ is not necessarily continuous in ${\bf y}_{n_{i-1}}$.
Indeed, for any positive measurable function $\varphi ({\bf y}_{n_{i-1}})$,
$h_i({\bf y}_{n_i})\varphi ({\bf y}_{n_{i-1}})$ also satisfies (\ref{1.210}).
It is our future problem to study if there exists a continuous solution to (\ref{1.210}).
\end{remark}


{\color{black}
For a probability density function $f$ on $\mathbb{R}^d$ such that  $f(x)\log f(x)$ is $dx$--integrable, let
\begin{equation}
\mathcal{S}(f):=\int_{\mathbb{R}^d} f(x)\log f(x) dx.
\end{equation}
The following is  an example such that $V_i(\mu_i,\nu_i;\pi_{\mu_{i-1},\nu_{i-1}})$ is finite.
The proof is given in the Appendix for completeness.

\begin{example}\label{ex2.1}
Suppose that  (A0, i) and (A1) hold.
Suppose also that 
there exists  $i\in \{2,\cdots, k_0\}$ such that
there exists $C>0$ for which
\begin{align}\label{2.15}
C^{-1}\exp(-C|{\bf x}_{n_i}-{\bf y}_{n_i}|^2)\le &p_i({\bf x}_{n_i},{\bf y}_{n_i})\\
\le &C\exp(-C^{-1}|{\bf x}_{[n_{i-1}+1,n_i]}-{\bf y}_{[n_{i-1}+1,n_i]}|^2),\notag
\end{align}
for $ {\bf x}_{n_i},{\bf y}_{n_i}\in\mathbb{R}^{n_i}$, and such that 
$\mu_i$ and $\nu_i$ have the finite second moments and $\mathcal{S}(f_{\nu_i})$ is finite.
Then $V_i(\mu_i,\nu_i;\pi_{\mu_{i-1},\nu_{i-1}})$ is finite.
\end{example}

We discuss the measure on the path space constructed from Theorem \ref{thm2.1} by a simple example.
Let
\begin{align*}
g(t,z):=&\frac{1}{\sqrt{2\pi t}}\exp \left(-\frac{|z|^2}{2t}\right), \quad (t,z)\in (0,\infty)\times \mathbb{R},\\
h_{1}(t, x):=&\int_{\mathbb{R}}
h_1(z)g(1-t, z- x)dz,\quad 0\le t<1,x\in \mathbb{R},\\
 \overline h_{2}(t, y_1,y_2):=&\int_{\mathbb{R}}h_2(y_1,z)
g(1-t, z- y_2)dz,\quad 0\le t<1,y_1,y_2\in \mathbb{R}.
\end{align*}
Let $X=(X_1, X_2)$ be an $\mathbb{R}^2$--valued random variable such that
$PX^{-1}=\mu$,
$B=(B_1,B_2)$ be an $\mathbb{R}^2$--valued Brownian motion which is independent of $X$,
and $\pi(t)(\omega):=\omega (t), \omega\in C([0,1];\mathbb{R}^2)$.
\begin{example}
Suppose that $d=k_0=2$, that $p(x,y)=\prod_{i=1}^2 g(1,y_i-x_i), x,y\in \mathbb{R}^2$, and that (A1) holds.
Then 
$\pi_{\mu,\nu}$ induces a Borel probability measure on $C([0,1];\mathbb{R}^2)$:
 for $A\in \mathcal{B}(C([0,1];\mathbb{R}^2))$, let
 \begin{align*}
P_{\pi_{\mu,\nu}}(A)
:=&E\left[\frac{h_2(X+B(1))}{\overline h_{2}(0, X_1+B_1(1),X_2)}\frac{h_1(X_1+B_1(1))}{h_1(0,X_1)};X+B(\cdot)\in A\right].
\end{align*}
It is easy to see that $P_{\pi_{\mu,\nu}}$  is a probability law of a Bernstein process (see \cite{B32, 1-Jamison1974})
and is Markovian in the case where $\mu$ and $\nu$ are product measures on
$\mathbb{R}^2$, in which case $f_{\nu}(y_2|y_1)$, $h_2(y_1,y_2)$ and $\overline h_2(t, y_1,y_2)$ are independent of $y_1$.
$P_{\pi_{\mu,\nu}}\pi(t)^{-1}, t\in (0,1)$ has the following probability density:
for $z=(z_i)_{i=1}^2\in \mathbb{R}^2$,
 \begin{align*}
\int_{\mathbb{R}^2}\frac{1}{h_1(0,x_1)}\mu(dx)\prod_{i=1,2}g(t,z_i-x_i)
\int_{\mathbb{R}}\frac{\overline h_{2}(t, y, z_2)}{\overline h_{2}(0, y,x_2)}h_1(y)g_1(1-t,y-z_1)dy.
\end{align*}
It is our future problem to construct a theory of stochastic analysis for the Bernstein process defined as above.
\end{example}
}

\section{Lemmas}\label{sec:3}

In this section, we give technical lemmas.

For $i\ge 1$ and a Borel measurable function $\varphi:\mathbb{R}^{n_i}\rightarrow [0,\infty)$,
let
\begin{align*}
\mathcal{I}_i(\varphi)(y):=
&\int_{\mathbb{R}^{n_i}}\varphi(x)\mu_i(dx)p_i(x,y),\quad y\in \mathbb{R}^{n_i}\end{align*}
(see (\ref{1.16}) and (\ref{1.21.0}) for notation).
For $n\ge 1$ and a function $\phi:\mathbb{R}^n\rightarrow \mathbb{R}\cup\{\infty\}$, let
$$Dom (\phi):=\{x\in \mathbb{R}^{n}:\phi(x)<\infty\},\quad Dom (\mathcal{I}_i(\varphi)):=Dom (\mathcal{I}_i(\varphi)(\cdot)),
$$
where we omit the dependence on $n$.

The following lemma will be used to prove Lemmas \ref{lm3.2}, \ref{lm3.3}, and \ref{lm3.4}.
{\color{black}It can be proven from the well--known fact that the domain of a convex function is convex and a convex function is continuous in the interior of its domain.
We give the proof in the Appendix for completeness.}

\begin{lemma}\label{lm3.1}
Let $\varphi_i:\mathbb{R}^{n_i}\rightarrow [0,\infty), i=1,\cdots, k_0$ be Borel measurable functions.
(i) Suppose that  (A0, i, ii) hold. Then $Dom (\mathcal{I}_1(\varphi_1))$  is convex and $\mathcal{I}_1(\varphi_1)(\cdot)$  is continuous in the interior of its domain $Dom (\mathcal{I}_1(\varphi_1))$.
(ii) Suppose that  (A0, i, iii) hold. Then for $i=2,\cdots, k_0$ and ${\bf y}_{n_{i-1}}\in \mathbb{R}^{n_{i-1}}$, 
$Dom (\mathcal{I}_i(\varphi_i)({\bf y}_{n_{i-1}},\cdot))$  is convex and $\mathcal{I}_i(\varphi_i)({\bf y}_{n_{i-1}},\cdot)$  is continuous in the interior of its domain $Dom (\mathcal{I}_i(\varphi_i)({\bf y}_{n_{i-1}},\cdot))$.
\end{lemma}

{\color{black}
The following lemma can be proven in the same way as Lemma \ref{lm3.4}.
We omit the proof.}

\begin{lemma}\label{lm3.2}
Suppose that (A0, i, ii) and (A1) hold. 
Then $h_1$ in (\ref{1.19}) can be taken to be continuous in $\mathbb{R}^{d_1}$.
In particular,
\begin{equation}
f_{\nu_1}(y)=h_1(y)\int_{\mathbb{R}^{d_1}}
\frac{\mu_1(dx)p_{1}(x,y)}
{\int_{\mathbb{R}^{d_1}}h_1(z)p_{1}(x,z)dz}=h_1(y)
\mathcal{I}_1\left(\frac{1}{h_1(0,\cdot)}\right)(y),
\quad y\in \mathbb{R}^{d_1},\label{3.3}
\end{equation}
\begin{align}\label{3.4.0}
\pi_{\mu_1,\nu_1}(y,dx)
=&\frac{p_{1}(x,y)}{\mathcal{I}_1\left(\frac{1}{h_1(0,\cdot)}\right)(y)h_1(0,x)}
\mu_1(dx)\in\mathcal{P}(\mathbb{R}^{d_1}),\quad y\in f_{\nu_1}^{-1}((0,\infty)).
\end{align}
\end{lemma}

The following lemma plays a crucial role in the proof of Theorem \ref{thm2.1} and can be proven by
Lemma \ref{lm3.1} (see (\ref{2.3.1223})  for notation).

\begin{lemma}\label{lm3.3}
Suppose that (A0, i, ii), and (A1) hold.
Then for $h_1\in C(\mathbb{R}^{d_1})$ in Lemma \ref{lm3.2}, 
$f_{\nu_1}^{-1}((0,\infty))\ni y\mapsto 
\pi_{\mu_1,\nu_1}\otimes\mu_{2|1}(y, dx)\in \mathcal{P}(\mathbb{R}^{n_2})$
is weakly continuous, i.e.,
for any $\varphi\in C_0(\mathbb{R}^{n_2})$ and $y\in \mathbb{R}^{d_1}$ such that $f_{\nu_1}(y)>0$,
\begin{equation}
\lim_{z\to y}\int_{\mathbb{R}^{n_2}}\varphi(x)
\pi_{\mu_1,\nu_1}\otimes\mu_{2|1}(z, dx)
=\int_{\mathbb{R}^{n_2}}\varphi(x)\pi_{\mu_1,\nu_1}\otimes\mu_{2|1}(y, dx).
\end{equation}
\end{lemma}

\begin{proof}
We only have to consider the case where $\varphi\not\equiv 0$ and $\varphi\ge 0$.
\begin{align}
&\int_{\mathbb{R}^{n_2}}\varphi(x)\pi_{\mu_1,\nu_1}\otimes\mu_{2|1}(z, dx)\\
=&\frac{1}{\mathcal{I}_1\left(\frac{1}{h_1(0,\cdot)}\right)(z)}
\mathcal{I}_1\left(\frac{1}{h_1(0,\cdot)}\int_{\mathbb{R}^{d_2}}\varphi(\cdot,x)\mu_2(dx|\cdot)\right)(z),\quad z\in f_{\nu_1}^{-1}((0,\infty)).\notag
\end{align}
Indeed, from  (\ref{2.3.1223}) 
 and (\ref{3.4.0}),
$$\pi_{\mu_1,\nu_1}\otimes\mu_{2|1}(z, d{\bf x}_{n_2})
=\frac{p_1({\bf x}_{n_1},z)}{h_1(0,{\bf x}_{n_1})\mathcal{I}_1\left(\frac{1}{h_1(0,\cdot)}\right)(z)}\mu_1(d{\bf x}_{n_1})\mu_2(d{\bf x}_{[n_{1}+1,n_2]}|{\bf x}_{n_1}).
$$
From (\ref{3.3}) and the boundedness of $\varphi$,
\begin{equation}\label{3.81}
f_{\nu_1}^{-1}((0,\infty))\subset 
Dom\left(\mathcal{I}_1\left(\frac{1}{h_1(0,\cdot)}\right)\right)
\subset
Dom\left(\mathcal{I}_1\left(\frac{1}{h_1(0,\cdot)}\int_{\mathbb{R}^{d_2}}\varphi(\cdot,x)\mu_2(dx|\cdot)\right)\right).
\end{equation}
Since $f_{\nu_1}^{-1}((0,\infty))$ is an open set from (A1), the proof is over from Lemma \ref{lm3.1}, (i).
\end{proof}

{\color{black}Lemmas \ref{lm3.4} and \ref{lm3.7} play a crucial role in the proof of Theorem \ref{thm2.1}.
Since Theorem \ref{thm2.1} will be proven after we prove Proposition \ref{pp2.1}, 
we suppose that Proposition \ref{pp2.1} holds in Lemmas \ref{lm3.4} and \ref{lm3.7}.}

We recall  (\ref{3.12})--(\ref{2.5}) for notation. For $({\bf x}_{n_i},{\bf y}_{n_{i-1}})\in \mathbb{R}^{n_i}\times \mathbb{R}^{n_{i-1}}$, let
\begin{align}
\phi_i({\bf x}_{n_i},{\bf y}_{n_{i-1}}):=&
\frac{1}{h_i(0, {\bf x}_{n_i}, {\bf y}_{n_{i-1}})}\prod_{j=1}^{i-1}\frac{h_j ({\bf y}_{n_j})p_j({\bf x}_{n_j},{\bf y}_{n_j})}{h_j (0, {\bf x}_{n_j}, {\bf y}_{n_{j-1}})}.
\end{align}

\begin{lemma}\label{lm3.4}
Suppose that  (A0)--(A1) hold and that Proposition \ref{pp2.1} holds.
Then for $i=2,\cdots,k_0$ and ${\bf y}_{n_{i-1}}\in \mathbb{R}^{n_{i-1}}$ such that $f_{\nu_{i-1}}({\bf y}_{n_{i-1}})>0$, 
there exists a solution $h_i({\bf y}_{n_{i-1}}, \cdot)$ of (\ref{2.10}) such that $h_i({\bf y}_{n_{i-1}}, \cdot)\in C(\mathbb{R}^{d_i})$
and that
\begin{equation}\label{3.110}
f_{\nu_i}({\bf y}_{n_{i-1}}, y)=h_i({\bf y}_{n_{i-1}}, y)\mathcal{I}_i(\phi_i(\cdot,{\bf y}_{n_{i-1}}))({\bf y}_{n_{i-1}}, y),\quad y\in \mathbb{R}^{d_i}.
\end{equation}
In particular, $\pi_{\mu_{i},\nu_{i}}({\bf y}_{n_i}, dx)\in \mathcal{P}(\mathbb{R}^{n_{i}})$ for ${\bf y}_{n_i}$ such that
$f_{\nu_{i}}({\bf y}_{n_i})>0$.

\end{lemma}

\begin{proof}
We prove (\ref{3.110}) by induction.
From 
Lemma \ref{lm3.2}, there exists a continuous solution $h_1$ of (\ref{1.19}) such that 
for ${\bf y}_{n_1}\in \mathbb{R}^{n_{1}}$ for which $f_{\nu_1}({\bf y}_{n_1})>0$, 
$\pi_{\mu_{1},\nu_{1}}({\bf y}_{n_1}, dx)\in \mathcal{P}(\mathbb{R}^{n_{1}})$. 
Suppose that $\pi_{\mu_{i-1},\nu_{i-1}}({\bf y}_{n_{i-1}}, dx)\in \mathcal{P}(\mathbb{R}^{n_{i-1}})$ for ${\bf y}_{n_{i-1}}\in \mathbb{R}^{n_{i-1}}$ such that $f_{\nu_{i-1}}({\bf y}_{n_{i-1}})>0$. 
Then there exists a solution $h_i({\bf y}_{n_{i-1}}, \cdot)$  of (\ref{2.10}) from Proposition \ref{pp2.1}
and the equality in (\ref{3.110}) holds $dy{\rm -a.e.}$ on $\mathbb{R}^{d_i}$ (see (\ref{2.3.1223}) for notation). 
The following also holds: 
\begin{equation}\label{3.10}
\mathcal{I}_i(\phi_i(\cdot,{\bf y}_{n_{i-1}}))( {\bf y}_{n_{i-1}},y)>0,\quad y\in\mathbb{R}^{d_i},
\end{equation}
since $p$ is positive.
Otherwise, $\phi_i({\bf x}_{n_i},{\bf y}_{n_{i-1}})=0$, $\mu_i(d{\bf x}_{n_i})$--a.e., which implies
that 
$$\mathcal{I}_i(\phi_i(\cdot,{\bf y}_{n_{i-1}}))( {\bf y}_{n_{i-1}},y)=0,\quad y\in\mathbb{R}^{d_i},
$$
$f_{\nu_i}({\bf y}_{n_{i-1}},y)=0$, $dy{\rm -a.e.}$  on $\mathbb{R}^{d_i}$,
 and hence $f_{\nu_{i-1}}({\bf y}_{n_{i-1}})=0$.

Let
\begin{equation}\label{3.11}
\overline h_i({\bf y}_{n_{i-1}},y):=\frac{f_{\nu_i}({\bf y}_{n_{i-1}},y)}{\mathcal{I}_i(
\phi_i(\cdot,{\bf y}_{n_{i-1}}))( {\bf y}_{n_{i-1}},y)}, \quad 
y\in \mathbb{R}^{d_i}.
\end{equation}
We show that (\ref{3.110}) with $h_i$ replaced by $\overline h_i$ holds, which implies that $\pi_{\mu_{i},\nu_{i}}({\bf y}_{n_i}, dx)\in \mathcal{P}(\mathbb{R}^{n_{i}})$ for ${\bf y}_{n_i}\in \mathbb{R}^{n_{i}}$ such that $f_{\nu_{i}}({\bf y}_{n_i})>0$. 
Since (\ref{3.110}) holds $dy{\rm -a.e.}$ on $\mathbb{R}^{d_i}$,
$$\overline h_i({\bf y}_{n_{i-1}},y)=h_i({\bf y}_{n_{i-1}},y), \quad dy\hbox{--a.e. on }Dom (\mathcal{I}_i(\phi_i(\cdot,{\bf y}_{n_{i-1}}))( {\bf y}_{n_{i-1}},\cdot)),$$
$$\overline h_i({\bf y}_{n_{i-1}},y)=0=h_i({\bf y}_{n_{i-1}}, y), \quad dy\hbox{--a.e. on }Dom (\mathcal{I}_i(\phi_i(\cdot,{\bf y}_{n_{i-1}}))( {\bf y}_{n_{i-1}},\cdot))^c,$$
which implies that 
$$\phi_i({\bf x}_{n_i},{\bf y}_{n_{i-1}})=\overline \phi_i({\bf x}_{n_i},{\bf y}_{n_{i-1}}),\quad ({\bf x}_{n_i},{\bf y}_{n_{i-1}})\in \mathbb{R}^{n_i}\times\mathbb{R}^{n_{i-1}},$$
where $\overline \phi_i$ denotes $\phi_i$ with $h_i$ replaced by $\overline h_i$ (see (\ref{3.12})).
From (\ref{3.11}),
$\overline h_i({\bf y}_{n_{i-1}},\cdot)$ satisfies the equality in (\ref{3.110}) with $h_i=\overline h_i$ on  $Dom (\mathcal{I}_i(\phi_i(\cdot,{\bf y}_{n_{i-1}}))( {\bf y}_{n_{i-1}},\cdot))$.
From (\ref{3.11}), 
\begin{equation}
f_{\nu_i}({\bf y}_{n_{i-1}},y)=\overline h_i({\bf y}_{n_{i-1}},y)=0,\quad y\in Dom (\mathcal{I}_i(\phi_i(\cdot,{\bf y}_{n_{i-1}}))( {\bf y}_{n_{i-1}},\cdot))^c,
\end{equation}
since (\ref{3.110}) holds $dy{\rm -a.e.}$   on $\mathbb{R}^{d_i}$, and 
$$f_{\nu_i}({\bf y}_{n_{i-1}},y)=h_i({\bf y}_{n_{i-1}},y)=0 \quad dy\hbox{--a.e. on }Dom (\mathcal{I}_i(\phi_i(\cdot,{\bf y}_{n_{i-1}}))( {\bf y}_{n_{i-1}},\cdot))^c$$
and since $f_{\nu_i}$ is continuous from (A1).

In the rest of the proof, we replace $h_i$  by $\overline h_i$ in  (\ref{3.110})--(\ref{3.10}), and show that  $\overline h_i({\bf y}_{n_{i-1}},\cdot)\in C(\mathbb{R}^{d_i})$ for ${\bf y}_{n_{i-1}}\in \mathbb{R}^{n_{i-1}}$ such that $f_{\nu_{i-1}}({\bf y}_{n_{i-1}})>0$.
First, we  show that $\overline  h_i({\bf y}_{n_{i-1}}, \cdot)$ is continuous in $f_\nu({\bf y}_{n_{i-1}}, \cdot)^{-1}((0,\infty))$.
From  (\ref{3.110}), 
\begin{equation}
f_{\nu_i}({\bf y}_{n_{i-1}}, \cdot)^{-1}((0,\infty))\subset Dom (\mathcal{I}_i(\overline \phi_i(\cdot,{\bf y}_{n_{i-1}}))( {\bf y}_{n_{i-1}},\cdot)).
\end{equation}
From (A1), $ f_{\nu_i}({\bf y}_{n_{i-1}}, \cdot)^{-1}((0,\infty))$ is an open set.
$\mathcal{I}_i(\overline \phi_i(\cdot,{\bf y}_{n_{i-1}}))( {\bf y}_{n_{i-1}},\cdot)$ is continuous in the interior of $Dom (\mathcal{I}_i(\overline \phi_i(\cdot,{\bf y}_{n_{i-1}}))( {\bf y}_{n_{i-1}},\cdot))$ from Lemma \ref{lm3.1}, (ii).
In particular, from (\ref{3.11}), $\overline  h_i({\bf y}_{n_{i-1}}, \cdot)$ is continuous in $f_\nu({\bf y}_{n_{i-1}}, \cdot)^{-1}((0,\infty))$.

If $f_{\nu_i}({\bf y}_{n_{i-1}}, y)=0$, 
then $\overline  h_i({\bf y}_{n_{i-1}},y)=0$ from (\ref{3.11}).
Let $\mathbb{R}^{d_i}\ni y_{n}\to y, n\to\infty$. 
The following together with (\ref{3.10}) completes the proof: from (A1) and  (\ref{3.110}), by Fatou's lemma,
\begin{align}
&\lim_{n\to\infty}f_{\nu_i}({\bf y}_{n_{i-1}},y_{n})=f_{\nu_i}({\bf y}_{n_{i-1}},y)=0\\
\ge &\limsup_{n\to\infty} \overline h_i({\bf y}_{n_{i-1}},y_n)
\times \liminf_{n\to\infty} \mathcal{I}_i(\overline \phi_i(\cdot,{\bf y}_{n_{i-1}}))({\bf y}_{n_{i-1}},y_n)\notag\\
\ge &\limsup_{n\to\infty} \overline h_i({\bf y}_{n_{i-1}},y_n)
\times \mathcal{I}_i(\overline \phi_i(\cdot,{\bf y}_{n_{i-1}}))({\bf y}_{n_{i-1}}, y).\notag
\end{align}
\end{proof}

{\color{black}
We recall that for a metric space $S$, ${P}(S)$ is endowed with weak topology.
Since ${\bf m}(dx\hbox{ }dy)$ in Theorem \ref{thm1.1} is uniquely determined by $(q,{\bf m}_1,{\bf m}_2)$, we
write ${\bf m}(dx\hbox{ }dy)={\bf m}(dx\hbox{ }dy;q,{\bf m}_1,{\bf m}_2)$.
The following lemma is made use of in the proof of Lemma \ref{lm3.7}
and is given for the sake of readers' convenience.

\begin{lemma}[see \cite{mikami2021}, Theorem 2.1]\label{lm3.6.0}
Suppose that $S$ is a complete $\sigma$-compact metric space, and that 
$q, q_n\in C(S \times S;(0,\infty))$, ${\bf m}_i, {\bf m}_{i,n}\in \mathcal {P}(S)$, 
$n\ge 1$,  $i=1,2$ and 
\begin{align*} 
\lim_{n\to\infty}q_n=&q,\quad \hbox{locally uniformly},\\
\lim_{n\to\infty}{\bf m}_{1,n}\times {\bf m}_{2,n}=&{\bf m}_{1}\times {\bf m}_{2},\quad \hbox{weakly.}
\end{align*}
Then 
\begin{equation*} 
\lim_{n\to\infty}{\bf m}(dx\hbox{ }dy;q_n,{\bf m}_{1,n},{\bf m}_{2,n})={\bf m}(dx\hbox{ }dy;q,{\bf m}_1,{\bf m}_2),\quad \hbox{weakly.}
\end{equation*}
\end{lemma}
}


\begin{lemma}\label{lm3.7}
Suppose that  (A0)--(A1) hold and that Proposition \ref{pp2.1} holds.
Then for $h_1\in C(\mathbb{R}^{d_1})$ in Lemma \ref{lm3.2}, 
\begin{equation}\label{3.25.1222}
f_{\nu_1}^{-1}((0,\infty))\ni {\bf y}_{n_1}\mapsto \pi_{\mu_2,\nu_2}({\bf y}_{n_1},dx\hbox{ }dy)\in \mathcal{P}(\mathbb{R}^{n_2}\times \mathbb{R}^{d_2})
\end{equation}
is weakly continuous (see (\ref{2.12weakly continuous}) for definition). 
For $i=2,\cdots, k_0$ and $h_i$ in Lemma \ref{lm3.4},
the following is 
measurable:
for a bounded Borel measurable function $\varphi:\mathbb{R}^{n_i}\times \mathbb{R}^{n_i}\rightarrow\mathbb{R}$,
\begin{equation}\label{3.19}
f_{\nu_{i-1}}^{-1}((0,\infty))\ni {\bf y}_{n_{i-1}}\mapsto 
\int_{\mathbb{R}^{n_i}\times \mathbb{R}^{d_i}}\varphi(x,({\bf y}_{n_{i-1}},y))\pi_{\mu_i,\nu_i}({\bf y}_{n_{i-1}},dx\hbox{ }dy).
\end{equation}
\end{lemma}

\begin{proof}
We prove this lemma by induction.
From Lemma \ref{lm3.3}, the following is weakly continuous (see (\ref{2.3.1223})):
$$f_{\nu_{1}}^{-1}((0,\infty))\ni {\bf y}_{n_1}\mapsto \pi_{\mu_{1},\nu_{1}}({\bf y}_{n_1}, dx)
\in \mathcal{P}(\mathbb{R}^{n_1}).$$
Suppose that  the following is weakly 
measurable:
$$f_{\nu_{i-1}}^{-1}((0,\infty))\ni {\bf y}_{n_{i-1}}\mapsto \pi_{\mu_{i-1},\nu_{i-1}}({\bf y}_{n_{i-1}}, dx)\in \mathcal{P}(\mathbb{R}^{n_{i-1}}),$$
i.e., the following is measurable:  for any $\varphi \in C_0 (\mathbb{R}^{n_{i-1}})$,
\begin{align}\label{3.15}
f_{\nu_{i-1}}^{-1}((0,\infty))\ni{\bf y}_{n_{i-1}}\mapsto &\int_{\mathbb{R}^{n_{i-1}}}\varphi (x)\pi_{\mu_{i-1},\nu_{i-1}}({\bf y}_{n_{i-1}}, dx).
\end{align}
Let $q_{i-1}$ be a positive continuous probability density on $\mathbb{R}^{n_{i-1}}$.
We prove that in $f_{\nu_{i-1}}^{-1}((0,\infty))$, 
the following is continuous  when $i=2$ and is 
measurable  when $i\ne 2$:
for any $\varphi \in C_0 (\mathbb{R}^{n_i}\times \mathbb{R}^{n_i})$,
\begin{align}\label{3.21}
{\bf y}_{n_{i-1}}\mapsto &\int_{\mathbb{R}^{n_{i-1}}\times \mathbb{R}^{n_i}\times \mathbb{R}^{d_i}}
\varphi (x, (z,y))
q_{i-1}(z)dz\pi_{\mu_i,\nu_i}({\bf y}_{n_{i-1}},dx\hbox{ }dy).
\end{align}
From (\ref{2.10})--(\ref{2.7.1}) and Lemmas \ref{lm3.2} and \ref{lm3.4},  for ${\bf y}_{n_{i-1}}\in \mathbb{R}^{n_{i-1}}$ such that $f_{\nu_{i-1}}({\bf y}_{n_{i-1}})>0$, the following holds (see (\ref{3.12}) for notation):
\begin{align}
&q_{i-1}({\bf z}_{n_{i-1}})d{\bf z}_{n_{i-1}}\pi_{\mu_i,\nu_i}({\bf y}_{n_{i-1}},d{\bf x}_{n_i}\hbox{ }d{\bf y}_{[n_{i-1}+1,n_i]})\notag\\
=&
\frac{h_i({\bf y}_{n_{i-1}},{\bf y}_{[n_{i-1}+1,n_i]})q_{i-1}({\bf z}_{n_{i-1}}) p({\bf x}_{n_i},({\bf y}_{n_{i-1}},{\bf y}_{[n_{i-1}+1,n_i]}))d{\bf z}_{n_{i-1}}d{\bf y}_{[n_{i-1}+1,n_i]}}{\int_{\mathbb{R}^{n_{i-1}}\times \mathbb{R}^{d_i}}h_i({\bf y}_{n_{i-1}}, y)q_{i-1}(z) p({\bf x}_{n_i},({\bf y}_{n_{i-1}},y))dzdy}\notag\\
&\qquad\times\pi_{\mu_{i-1},\nu_{i-1}}\otimes\mu_{i|i-1}({\bf y}_{n_{i-1}}, d{\bf x}_{n_i})\notag\\
\in &\mathcal{A}(\pi_{\mu_{i-1},\nu_{i-1}}\otimes\mu_{i|i-1}({\bf y}_{n_{i-1}}, d{\bf x}_{n_i}), 
q_{i-1}({\bf z}_{n_{i-1}})f_{\nu_{i}}({\bf y}_{[n_{i-1}+1,n_i]}|{\bf y}_{n_{i-1}})d{\bf z}_{n_{i-1}}d{\bf y}_{[n_{i-1}+1,n_i]}).\label{3.270}
\end{align}

For any $n\ge 1$, there exists a closed set
$F_{i-1,n}\subset \mathbb{R}^{n_{i-1}}$ with the Lebesgue measure $|F_{i-1,n}^c|\le n^{-1}$ such that 
the following is continuous on $f_{\nu_{i-1}}^{-1}((0,\infty))\cap F_{i-1,n}$:
for any $\varphi\in C_0(\mathbb{R}^{n_i})$,
\begin{equation}\label{3.26}
{\bf y}_{n_{i-1}}\mapsto \int_{\mathbb{R}^{n_i}}\varphi(x)\pi_{\mu_{i-1},\nu_{i-1}}\otimes\mu_{i|i-1}({\bf y}_{n_{i-1}}, dx).
\end{equation}
Indeed, for any $\varphi\in C_0(\mathbb{R}^{n_i})$,
(\ref{3.26})  is 
measurable on $f_{\nu_{i-1}}^{-1}((0,\infty))$ from 
the assumption of induction.
By Lusin's theorem, for any countable set $S\subset C_0(\mathbb{R}^{n_i})$, there exists a closed set $F_{i-1,n}$ with the Lebesgue measure $|F_{i-1,n}^c|\le n^{-1}$ such that (\ref{3.26}) is continuous on $f_{\nu_{i-1}}^{-1}((0,\infty))\cap F_{i-1,n}$ for all $\varphi\in S$.
The space of continuous functions on a compact subset of a Euclidean space is separable and a Euclidean space is $\sigma$--compact.
\begin{equation}
f_{\nu_{i-1}}^{-1}((0,\infty))\ni {\bf y}_{n_{i-1}}\mapsto q_{i-1}({\bf z}_{n_{i-1}})d{\bf z}_{n_{i-1}}f_{\nu_{i}}({\bf y}_{[n_{i-1}+1,n_i]}|{\bf y}_{n_{i-1}})d{\bf y}_{[n_{i-1}+1,n_i]}
\end{equation}
is weakly continuous from (A1).
\begin{equation}\label{3.31}
f_{\nu_{i-1}}^{-1}((0,\infty))\ni {\bf y}_{n_{i-1}}\mapsto q_{i-1}({\bf z}_{n_{i-1}}) p({\bf x}_{n_i},({\bf y}_{n_{i-1}},{\bf y}_{[n_{i-1}+1,n_i]}))
\end{equation}
is positive and is continuous locally uniformly in $({\bf x}_{n_i},({\bf z}_{n_{i-1}},{\bf y}_{[n_{i-1}+1,n_i]}))$ from (A0, i).
From (\ref{3.270})--(\ref{3.31}) and Lemma \ref{lm3.6.0}, 
\begin{equation}
f_{\nu_{i-1}}^{-1}((0,\infty))\cap F_{i-1,n}\ni{\bf y}_{n_{i-1}}\mapsto q_{i-1}({\bf z}_{n_{i-1}})d{\bf z}_{n_{i-1}}\pi_{\mu_i,\nu_i}({\bf y}_{n_{i-1}},d{\bf x}_{n_i}\hbox{ }d{\bf y}_{[n_{i-1}+1,n_i]})
\end{equation}
is weakly continuous for all $n\ge 1$, which implies that (\ref{3.21}) is 
measurable. 
$F_{1,n}=\mathbb{R}^{d_1}$ above, from Lemma \ref{lm3.3}.
In particular, (\ref{3.21}) is continuous when $i=2$.

If $\varphi({\bf x}_{n_i},({\bf y}_{n_{i-1}},{\bf y}_{[n_{i-1}+1,n_i]}))=\overline \varphi_1({\bf y}_{n_{i-1}})\overline \varphi_2({\bf x}_{n_i},{\bf y}_{[n_{i-1}+1,n_i]})$ for Borel measurable functions $\overline \varphi_1:\mathbb{R}^{n_{i-1}}\rightarrow \mathbb{R},\overline \varphi_2:\mathbb{R}^{n_i}\times \mathbb{R}^{d_i}\rightarrow \mathbb{R}$,
then it is easy to see that (\ref{3.19}) is 
measurable.
The proof of 
measurability of (\ref{3.19})
is easily done by the monotone class theorem and the monotone convergence theorem.

We prove that the following is weakly 
measurable (see (\ref{3.15}) for definition):
$$f_{\nu_{i}}^{-1}((0,\infty))\ni {\bf y}_{n_i}\mapsto \pi_{\mu_{i},\nu_{i}}({\bf y}_{n_i}, dx)\in \mathcal{P}(\mathbb{R}^{n_i}).$$
Let 
\begin{align*}
d (z, f_{\nu_{i}}^{-1}(\{0\})):=&\inf\{|z-w|: f_{\nu_{i}}(w)=0,
w\in\mathbb{R}^{n_i}\},\quad z\in\mathbb{R}^{n_i},\\
U_{-2/n}(f_{\nu_{i}}^{-1}((0,\infty))):=&
\{z\in f_{\nu_{i}}^{-1}((0,\infty)):d (z, f_{\nu_{i}}^{-1}(\{0\}))> 2/n\},\quad n\ge 1.
\end{align*}
Since $f_{\nu_{i}}^{-1}((0,\infty))$ is open from (A1),
$$f_{\nu_{i}}^{-1}((0,\infty))=\bigcup_{n\ge 1}U_{-2/n}(f_{\nu_{i}}^{-1}((0,\infty))).$$
Take probability densities $r_n\in C_0(\mathbb{R}^{d_i};[0,\infty))$ such that 
$r_n(x)=0, |x|\ge n^{-1}$ and that $r_n(x)dx$ weakly converges to a delta measure on $\{0\}\subset\mathbb{R}^{d_i}$, as $n\to\infty$.
Foy any $\varphi \in C_0(\mathbb{R}^{n_i})$ and $n\ge 1$, the following is 
measurable in 
${\bf y}_{n_i}\in U_{-2/n}(f_{\nu_{i}}^{-1}((0,\infty)))$:
\begin{align}\label{3.34}
&
\int_{\mathbb{R}^{n_i}\times \mathbb{R}^{d_i}}\varphi(x)
r_n({\bf y}_{[n_{i-1}+1,n_i]}-y)
\frac{f_{\nu_{i-1}}({\bf y}_{n_{i-1}})}{f_{\nu_{i}}({\bf y}_{n_{i-1}},y)}
\pi_{\mu_i,\nu_i}({\bf y}_{n_{i-1}},dx\hbox{ }dy).
\end{align}
Indeed, for $
{\bf y}_{[n_{i-1}+1,n_i]}\in \cup_{{\bf y}\in \mathbb{R}^{n_{i-1}}}\{
y\in\mathbb{R}^{d_{i}}:({\bf y},y)\in U_{-2/n}(f_{\nu_{i}}^{-1}((0,\infty)))\}$,  (\ref{3.34}) is 
measurable in ${\bf y}_{n_{i-1}}\in
\{y\in\mathbb{R}^{n_{i-1}}:
(y,{\bf y}_{[n_{i-1}+1,n_i]})\in U_{-2/n}(f_{\nu_{i}}^{-1}((0,\infty)))\}$ that is an open subset of $f_{\nu_{i-1}}^{-1}((0,\infty))$ from (A1), since (\ref{3.19}) is 
measurable.
For ${\bf y}_{n_{i-1}}\in \cup_{y\in \mathbb{R}^{d_{i}}}
\{{\bf y}\in\mathbb{R}^{n_{i-1}}:({\bf y},y)\in U_{-2/n}(f_{\nu_{i}}^{-1}((0,\infty)))\}$,
(\ref{3.34}) is also continuous in ${\bf y}_{[n_{i-1}+1,n_i]}\in \{y\in\mathbb{R}^{d_{i}}:
({\bf y}_{n_{i-1}},y)\in U_{-2/n}(f_{\nu_{i}}^{-1}((0,\infty)))\}$, 
from (A1) by the bounded convergence theorem, since the supports of $\varphi$ and $r_n$ are bounded.

For a set $A$, let
$$
1_A(x):=
\begin{cases}
1, &x\in A,\\
0, &x\not\in A.
\end{cases}
$$
The following is 
measurable in ${\bf y}_{n_i}\in f_{\nu_{i}}^{-1}((0,\infty))$ from the discussion above:
\begin{align}\label{3.35}
&1_{U_{-2/n}(f_{\nu_{i}}^{-1}((0,\infty)))}({\bf y}_{n_i})
\int_{\mathbb{R}^{n_i}\times \mathbb{R}^{d_i}}\varphi(x)
r_n({\bf y}_{[n_{i-1}+1,n_i]}-y)
\frac{f_{\nu_{i-1}}({\bf y}_{n_{i-1}})}{f_{\nu_{i}}({\bf y}_{n_{i-1}},y)}
\pi_{\mu_i,\nu_i}({\bf y}_{n_{i-1}},dx\hbox{ }dy).
\end{align}
As $n\to\infty$, by the bounded convergence theorem,
 (\ref{3.35}) converges to the following: for ${\bf y}_{n_i}\in f_{\nu_{i}}^{-1}((0,\infty))$,
 \begin{align}
&\int_{\mathbb{R}^{n_i}}\varphi(x)
\frac{f_{\nu_{i-1}}({\bf y}_{n_{i-1}})}{f_{\nu_{i}}({\bf y}_{n_i})}
\frac{h_i ({\bf y}_{n_i})p_i(x,{\bf y}_{n_i})}{h_i (0,x, {\bf y}_{n_{i-1}})}
\pi_{\mu_{i-1},\nu_{i-1}}\otimes\mu_{i|i-1}({\bf y}_{n_{i-1}},dx)\\
=&\int_{\mathbb{R}^{n_i}}\varphi(x)\pi_{\mu_{i},\nu_{i}}({\bf y}_{n_i},dx)\notag
\end{align}
(see (\ref{2.5}) and (\ref{2.7.1})), which is 
measurable in ${\bf y}_{n_i}\in f_{\nu_{i}}^{-1}((0,\infty))$
as the limit of 
measurable functions.
Indeed, 
$h_i ({\bf y}_{n_{i}})>0,  {\bf y}_{n_i}\in f_{\nu_{i}}^{-1}((0,\infty))$ from (\ref{3.110}), and
$h_i ({\bf y}_{n_{i-1}},\cdot)\in C(\mathbb{R}^{d_i})$ from Lemma \ref{lm3.4}, and
\begin{align*}
h_i (0, x, {\bf y}_{n_{i-1}})
=&\int_{\mathbb{R}^{d_i}}h_i({\bf y}_{n_{i-1}}, y)
p_{i}(x,({\bf y}_{n_{i-1}},y))dy\\
\ge &\int_{\{y\in \mathbb{R}^{d_i}:|y-{\bf y}_{[n_{i-1}+1,n_i]}|<1\}}
h_i({\bf y}_{n_{i-1}}, y)
p_{i}(x,({\bf y}_{n_{i-1}},y))dy, \quad x\in \mathbb{R}^{n_i},
\end{align*}
which is bounded  from below, in $x\in$supp$(\varphi)$, by a positive constant (see (\ref{3.12}) for notation).
The following are also bounded in $x\in$supp$(\varphi)$ for sufficiently large $n\ge 1$, and by the bounded convergence theorem, as $n\to\infty$,
\begin{align}
&\int_{\mathbb{R}^{d_i}}
r_n({\bf y}_{[n_{i-1}+1,n_i]}-y)
\frac{h_i ({\bf y}_{n_{i-1}},y)p_i(x,({\bf y}_{n_{i-1}},y))}
{f_{\nu_{i}}({\bf y}_{n_{i-1}},y)}dy\\
\to &\frac{h_i ({\bf y}_{n_{i-1}},{\bf y}_{[n_{i-1}+1,n_i]})p_i(x,({\bf y}_{n_{i-1}},{\bf y}_{[n_{i-1}+1,n_i]}))}
{f_{\nu_{i}}({\bf y}_{n_{i-1}},{\bf y}_{[n_{i-1}+1,n_i]})},\quad x\in \hbox{supp}(\varphi).\notag
\end{align}
\end{proof}

\section{Proof of main results}\label{sec:4}

In this section, we prove our results.

{\color{black}We briefly explain the idea of the proof.
Most parts of Proposition \ref{pp2.1} can be proven  using the known results.
It is Schr\"odinger's problems and functional equations for conditional distributions.
Proposition \ref{pp2.4} is proven in Lemmas \ref{lm3.2} and \ref{lm3.7}.
We explain the idea of the proof of Theorem \ref{thm2.1}.
For two solutions $h_i$ and  $\overline h_i $ of (\ref{2.10}),
there exists a function $\varphi({\bf y}_{n_{i-1}})$ such that 
$$h_i ({\bf y}_{n_{i}})=\varphi({\bf y}_{n_{i-1}})\overline h_i ({\bf y}_{n_{i}}),
\quad {\bf y}_{n_{i}}=({\bf y}_{n_{i-1}}, {\bf y}_{[n_{i-1}+1,n_i]})\in\mathbb{R}^{n_i},
$$
which implies that for any Borel 
measurable function $\xi: \mathbb{R}^{n_{i-1}}\rightarrow \mathbb{R}^{d_i}$,
\begin{equation}\label{4.1}
\frac{h_i ({\bf y}_{n_{i-1}}, {\bf y}_{[n_{i-1}+1,n_i]})}
{h_i ({\bf y}_{n_{i-1}}, \xi({\bf y}_{n_{i-1}}))}
=\frac{\overline h_i ({\bf y}_{n_{i-1}}, {\bf y}_{[n_{i-1}+1,n_i]})}
{\overline h_i ({\bf y}_{n_{i-1}}, \xi({\bf y}_{n_{i-1}}))}.
\end{equation}
(\ref{4.1}) is uniquely determined by $\xi$ and  is also a solution of (\ref{2.10}), provided the denominators are positive.
If $h_i $ is 
measurable, then so is (\ref{4.1}) even if $\overline h_i $ is not.
It led us to prove the 
measurability of (\ref{4.1}).
The denominators of  (\ref{4.1}) are positive if 
$f_{\nu_i}({\bf y}_{n_{i-1}}, \xi({\bf y}_{n_{i-1}}))>0$ from (\ref{2.10}).
We find a graph of Borel 
measurable function in $f_{\nu_i}^{-1}((0,\infty))$ 
which is $\sigma$--compact.
This is achieved  by the so--called selection lemma in control theory.
We prove the continuity and the 
measurability of (\ref{4.1}) in ${\bf y}_{[n_{i-1}+1,n_i]}$ for each ${\bf y}_{n_{i-1}}$ and 
in ${\bf y}_{n_{i-1}}$ for each ${\bf y}_{[n_{i-1}+1,n_i]}$, respectively.
The rest of the proof is standard once the measurability of $h_i$ is proven.
}

We prove  Proposition \ref{pp2.1}.

\begin{proof}[Proof of Proposition \ref{pp2.1}]
First, we prove the existence of a solution to (\ref{2.10}) that is unique up to a multiplicative function of ${\bf y}_{n_{i-1}}$.
Let $q_{i-1}$ be a positive continuous probability density function on  $\mathbb{R}^{n_{i-1}}$.
For ${\bf y}_{n_{i-1}}\in \mathbb{R}^{n_{i-1}}$ such that $f_{\nu_{i-1}}({\bf y}_{n_{i-1}})>0$,
consider the following SFE:
\begin{align}\label{1.21}
&q_{i-1}({\bf z}_{n_{i-1}})f_{\nu_{i}}({\bf y}_{[n_{i-1}+1,n_i]}|{\bf y}_{n_{i-1}})
d{\bf z}_{n_{i-1}}d{\bf y}_{[n_{i-1}+1,n_i]}\\
=&\tilde h_i({\bf y}_{n_{i-1}},{\bf z}_{n_{i-1}},{\bf y}_{[n_{i-1}+1,n_i]})d{\bf z}_{n_{i-1}}d{\bf y}_{[n_{i-1}+1,n_i]}\notag\\
&\qquad\times \int_{\mathbb{R}^{n_i}}
\frac{q_{i-1}({\bf z}_{n_{i-1}}) p(x,{\bf y}_{[n_{i-1}+1,n_i]}|{\bf y}_{n_{i-1}})
\pi_{\mu_{i-1},\nu_{i-1}}\otimes\mu_{i|i-1}({\bf y}_{n_{i-1}}, dx)}
{\int_{\mathbb{R}^{n_{i-1}}\times \mathbb{R}^{d_i}}\tilde h_i( {\bf y}_{n_{i-1}},z,y)q_{i-1}(z) 
p(x,y|{\bf y}_{n_{i-1}})dzdy}.\notag
\end{align}
Since 
$$\mathbb{R}^{n_i}\times\mathbb{R}^{n_{i-1}}\times \mathbb{R}^{d_i}\ni(x,z,y)
\mapsto q_{i-1}(z) p(x,y|{\bf y}_{n_{i-1}})$$
is positive and continuous,
there exists a  solution $\tilde h_i$ that is unique up to a multiplicative function of ${\bf y}_{n_{i-1}}$ (see 
Theorem \ref{thm1.1} in section \ref{sec:1}). 
Integrating the both sides of  (\ref{1.21}) in ${\bf z}_{n_{i-1}}$, 
\begin{equation}
h_i({\bf y}_{n_{i-1}},{\bf y}_{[n_{i-1}+1,n_i]}):=\int_{\mathbb{R}^{n_{i-1}}}\tilde h_i({\bf y}_{n_{i-1}},z,{\bf y}_{[n_{i-1}+1,n_i]}) q_{i-1}(z) dz
\end{equation}
is a solution to (\ref{2.10}).
A solution to (\ref{2.10}) is also that of (\ref{1.21}).
In particular, (\ref{2.10}) has a solution that is unique up to a multiplicative function of ${\bf y}_{n_{i-1}}$.
From (\ref{2.3.0})--(\ref{2.6.0}),
$\pi_{\mu_i, \nu_i }({\bf y}_{n_{i-1}},d{\bf x}_{n_i}\hbox{ }d{\bf y}_{[n_{i-1}+1,n_i]})$ is the unique minimizer of  (\ref{2.11}), 
$f_{\nu_{i-1}}({\bf y}_{n_{i-1}})d{\bf y}_{n_{i-1}}$--a.e.,
provided it is finite (see \cite{RT93}{\color{black}, Theorem 3}).
\end{proof}

We prove our main result.

\begin{proof}[Proof of Theorem \ref{thm2.1}]

From Lemma \ref{lm3.2},
there exists a continuous solution $h_1$ of (\ref{1.19}) such that
$\pi_{\mu_1,\nu_1}({\bf y}_{n_1}, dx)\in \mathcal{P}(\mathbb{R}^{n_1})$ for  
${\bf y}_{n_1}\in \mathbb{R}^{n_1}$ for which $f_{\nu_{1}}({\bf y}_{n_1})>0$.

From Lemma \ref{lm3.4}, for $i=2,\cdots, k_0$ and  ${\bf y}_{n_{i-1}}\in \mathbb{R}^{n_{i-1}}$ such that $f_{\nu_{i-1}}({\bf y}_{n_{i-1}})>0$,
(\ref{2.10}) has a solution $h_i({\bf y}_{n_{i-1}},\cdot)$ such that $h_i({\bf y}_{n_{i-1}},\cdot)\in C(\mathbb{R}^{d_i})$  and that $\pi_{\mu_i,\nu_i}({\bf y}_{n_i}, dx)\in \mathcal{P}(\mathbb{R}^{n_i})$
for  
${\bf y}_{n_i}\in \mathbb{R}^{n_i}$ for which $f_{\nu_{i}}({\bf y}_{n_i})>0$.
We construct a 
measurable function $\tilde h_i$ such that 
$\tilde h_i({\bf y}_{n_{i-1}},\cdot)\in C(\mathbb{R}^{d_i})$ for ${\bf y}_{n_{i-1}}\in \mathbb{R}^{n_{i-1}}$
and such that 
$$\tilde h_i({\bf y}_{n_{i-1}},\cdot)=h_i({\bf y}_{n_{i-1}},\cdot),\quad d{\bf y}_{n_{i-1}}\hbox{--a.e. on $f_{\nu_{i-1}}^{-1}((0,\infty))$},$$
up to a multiplicative function of ${\bf y}_{n_{i-1}}$ (see (\ref{4.5}) below).

Since $f_{\nu_i}$ is continuous, the set $f_{\nu_i}^{-1}((0,\infty))$ is open and hence is $\sigma$--compact.
In particular, by the selection lemma (see \cite{FR}, p. 199), there exists a Borel 
measurable function
$\xi_i:f_{\nu_{i-1}}^{-1}((0,\infty)) \rightarrow \mathbb{R}^{d_i}$
such that 
\begin{equation}
({\bf y}_{n_{i-1}},\xi_i({\bf y}_{n_{i-1}}))\in f_{\nu_i}^{-1}((0,\infty)),\quad d{\bf y}_{n_{i-1}}\hbox{--a.e. on }f_{\nu_{i-1}}^{-1}((0,\infty)).
\end{equation}
Here, notice that by the continuity of $f_{\nu_i}$,
$$\{{\bf y}_{n_{i-1}}\in \mathbb{R}^{n_{i-1}}:f_{\nu_i}({\bf y}_{n_{i-1}},\cdot)^{-1}((0,\infty))\ne\emptyset\}
=f_{\nu_{i-1}}^{-1}((0,\infty)).$$
We define
\begin{equation}
S_{\nu_{i-1},+}:=\{{\bf y}_{n_{i-1}}\in f_{\nu_{i-1}}^{-1}((0,\infty)): f_{\nu_i}({\bf y}_{n_{i-1}},\xi_i({\bf y}_{n_{i-1}}))>0
\}.
\end{equation}
It is easy to see that the function defined by 
\begin{equation}\label{4.5}
\tilde h_i({\bf y}_{n_{i-1}},{\bf y}_{[n_{i-1}+1,n_i]}):=1_{S_{\nu_{i-1},+}}({\bf y}_{n_{i-1}})\frac{h_i({\bf y}_{n_{i-1}},{\bf y}_{[n_{i-1}+1,n_i]})}{h_i({\bf y}_{n_{i-1}},\xi_i({\bf y}_{n_{i-1}}))}, \quad
{\bf y}_{n_{i}}\in \mathbb{R}^{{n_i}}
\end{equation}
also satisfies  (\ref{2.10}),  $f_{\nu_{i-1}}({\bf y}_{n_{i-1}})d{\bf y}_{n_{i-1}}\hbox{--a.e. on }f_{\nu_{i-1}}^{-1}((0,\infty))$
since 
\begin{align*}
1_{S_{\nu_{i-1},+}}({\bf y}_{n_{i-1}})=&1,\quad d{\bf y}_{n_{i-1}}\hbox{--a.e. on }f_{\nu_{i-1}}^{-1}((0,\infty)),\\
h_i({\bf y}_{n_{i-1}},\xi_i({\bf y}_{n_{i-1}}))>&0,\quad {\bf y}_{n_{i-1}}\in S_{\nu_{i-1},+} \hbox{ (from (\ref{3.110}))}.
\end{align*}
 
For ${\bf y}_{n_{i-1}}\in S_{\nu_{i-1},+}$, 
$\tilde h_i({\bf y}_{n_{i-1}},\cdot)\in C(\mathbb{R}^{d_i})$.
 For ${\bf y}_{n_{i-1}}\not\in S_{\nu_{i-1},+}$, 
$\tilde h_i({\bf y}_{n_{i-1}},\cdot)=0\in C(\mathbb{R}^{d_i})$.
To prove that $\tilde h_i({\bf y}_{n_{i-1}},{\bf y}_{[n_{i-1}+1,n_i]})$ is 
measurable, we prove that for ${\bf y}_{[n_{i-1}+1,n_i]}\in\mathbb{R}^{d_i}$, the following is 
measurable:
\begin{equation}
S_{\nu_{i-1},+}\ni {\bf y}_{n_{i-1}}\mapsto \tilde h({\bf y}_{n_{i-1}},{\bf y}_{[n_{i-1}+1,n_i]})\frac{f_{\nu_i}({\bf y}_{n_{i-1}},\xi_i({\bf y}_{n_{i-1}}))}{f_{\nu_{i-1}}({\bf y}_{n_{i-1}})},
\end{equation}
since 
\begin{align*}
&\{y\in\mathbb{R}^{n_{i-1}}|\tilde h_i(y,{\bf y}_{[n_{i-1}+1,n_i]})\ge r\}=\mathbb{R}^{n_{i-1}},\quad r\le 0,\\
&\{y\in\mathbb{R}^{n_{i-1}}|\tilde h_i(y,{\bf y}_{[n_{i-1}+1,n_i]})\ge r\}\\
=&\left\{y\in S_{\nu_{i-1},+}\biggl| \tilde h(y,{\bf y}_{[n_{i-1}+1,n_i]})\frac{f_{\nu_i}(y,\xi_i(y))}{f_{\nu_{i-1}}(y)}\ge r\frac{f_{\nu_i}(y,\xi_i(y))}{f_{\nu_{i-1}}(y)}\right\},\quad r>0.
\end{align*}
Recall that $f_{\nu_i}(y,\xi_i(y))$ and $f_{\nu_{i-1}}(y)$ are positive  for $y\in S_{\nu_{i-1},+}$
and are Borel measurable from (A1).

Take probability densities $r_n\in C_0(\mathbb{R}^{d_i};[0,\infty))$ such that 
$r_n(x)=0, |x|\ge n^{-1}$ and that $r_n(x)dx$ weakly
converges to a delta measure on $\{0\}\subset\mathbb{R}^{d_i}$ as $n\to\infty$.
From Lemma \ref{lm3.7}, the following is 
measurable in ${\bf y}_{n_{i-1}}$ on $S_{\nu_{i-1},+}$: 
\begin{align}\label{4.31}
&\int_{\mathbb{R}^{n_i}\times \mathbb{R}^{d_i}}
r_n({\bf y}_{[n_{i-1}+1,n_i]}-y)\frac{p_i(x,
({\bf y}_{n_{i-1}},\xi_i({\bf y}_{n_{i-1}})))}{p_i(x,({\bf y}_{n_{i-1}},y))}
{\pi}_{\mu_i,\nu_i}({\bf y}_{n_{i-1}},dx\hbox{ }dy)\\
=&\int_{\mathbb{R}^{n_i}}
\frac{1}{h_i(0, x, {\bf y}_{n_{i-1}})}
\pi_{\mu_{i-1},\nu_{i-1}}\otimes\mu_{i|i-1}({\bf y}_{n_{i-1}},dx)\notag\\
&\quad \times \int_{\mathbb{R}^{d_i}}
r_n({\bf y}_{[n_{i-1}+1,n_i]}-y)\frac{p_i(x,({\bf y}_{n_{i-1}},\xi_i({\bf y}_{n_{i-1}})))}{p_i(x,({\bf y}_{n_{i-1}},y))}
h_i({\bf y}_{n_{i-1}},y)p(x,({\bf y}_{n_{i-1}},y))dy\notag\\
=&\frac{1}{h_i({\bf y}_{n_{i-1}},\xi_i({\bf y}_{n_{i-1}}))}
\int_{\mathbb{R}^{d_i}}
r_n({\bf y}_{[n_{i-1}+1,n_i]}-y)h_i({\bf y}_{n_{i-1}},y)dy\notag\\
&\quad \times \int_{\mathbb{R}^{n_i}}
\frac{h_i({\bf y}_{n_{i-1}},\xi_i({\bf y}_{n_{i-1}}))p_i(x,({\bf y}_{n_{i-1}},\xi_i({\bf y}_{n_{i-1}})))}
{h_i(0, x, {\bf y}_{n_{i-1}})}
\pi_{\mu_{i-1},\nu_{i-1}}\otimes\mu_{i|i-1}({\bf y}_{n_{i-1}},dx)\notag\\
\to &\frac{h_i({\bf y}_{n_{i-1}},{\bf y}_{[n_{i-1}+1,n_i]})}{h_i({\bf y}_{n_{i-1}},\xi_i({\bf y}_{n_{i-1}}))}
\frac{f_{\nu_i}({\bf y}_{n_{i-1}},\xi_i({\bf y}_{n_{i-1}}))}{f_{\nu_{i-1}}({\bf y}_{n_{i-1}})},\quad n\to\infty\notag
\end{align}
from (\ref{3.110}), 
which is 
measurable in ${\bf y}_{n_{i-1}}$ on $S_{\nu_{i-1},+}$ as the limit of 
measurable functions
 (see (\ref{3.12}) and (\ref{2.10})--(\ref{2.7.1}) for notation). 
Indeed, since $h_i({\bf y}_{n_{i-1}},\cdot)$ is continuous, 
$$\sup\{h_i({\bf y}_{n_{i-1}},y)
:|{\bf y}_{[n_{i-1}+1,n_i]}-y|\le 1, y\in \mathbb{R}^{d_i}\}<\infty.$$

{\color{black}From (\ref{1.170}), (\ref{2.8}), (\ref{2.3}),  (\ref{2.6}), (\ref{1.27}), and (\ref{3.12}),  (\ref{2.10.1}) holds.
By induction, $\pi_{\mu_i,\nu_i}\in \mathcal{A}(\mu_i,\nu_i;\pi_{\mu_{i-1},\nu_{i-1}})$ from 
(\ref{2.8}), (\ref{2.5}), (\ref{2.3.1223}), (\ref{2.10}), and (\ref{2.10.1}), since
\begin{align}\label{4.9}
&\pi_{\mu_i,\nu_i}(d{\bf x}_{n_i}\hbox{ }d{\bf y}_{n_i})\notag\\
=&\frac{h_i ({\bf y}_{n_i})p_i({\bf x}_{n_i},{\bf y}_{n_i})}
{h_i(0, {\bf x}_{n_i}, {\bf y}_{n_{i-1}})}d{\bf y}_{[n_{i-1}+1,n_i]}
\pi_{\mu_{i-1},\nu_{i-1}}\otimes\mu_{i|i-1}(d{\bf x}_{n_i}\hbox{ }d{\bf y}_{n_{i-1}})\notag\\
=&f_{\nu_{i-1}}({\bf y}_{n_{i-1}})\frac{h_i ({\bf y}_{n_i})p_i({\bf x}_{n_i},{\bf y}_{n_i})}
{h_i(0, {\bf x}_{n_i}, {\bf y}_{n_{i-1}})}
\pi_{\mu_{i-1},\nu_{i-1}}\otimes\mu_{i|i-1}({\bf y}_{n_{i-1}},d{\bf x}_{n_i})d{\bf y}_i.
\end{align}
}


For any $\pi\in\mathcal{A}(\mu_i,\nu_i;\pi_{\mu_{i-1},\nu_{i-1}})$ such that 
$\pi\ll\pi_{0,i}$, from Remark \ref{rk2.2},
\begin{align}
&H(\pi\parallel\pi_{0,i})=
\int_{\mathbb{R}^{n_i}\times \mathbb{R}^{n_i}}\left\{\log \frac{d\pi}{d\pi_{0,i}}(x,y)\right\}\pi (dx\hbox{ }dy)\\
=&\int_{\mathbb{R}^{n_{i-1}}}f_{\nu_{i-1}}({\bf y}_{n_{i-1}})d{\bf y}_{n_{i-1}}\notag\\
&\qquad \times  H(\pi (d{\bf x}_{n_i}\hbox{ }d{\bf y}_{[n_{i-1}+1,n_i]}|{\bf y}_{n_{i-1}})\parallel\pi_{0,i} (d{\bf x}_{n_i}\hbox{ }d{\bf y}_{[n_{i-1}+1,n_i]}|{\bf y}_{n_{i-1}})),\notag
\end{align}
\begin{align}\label{4.4}
&H(\pi (d{\bf x}_{n_i}\hbox{ }d{\bf y}_{[n_{i-1}+1,n_i]}|{\bf y}_{n_{i-1}})\parallel\pi_{0,i} (d{\bf x}_{n_i}\hbox{ }d{\bf y}_{[n_{i-1}+1,n_i]}|{\bf y}_{n_{i-1}}))\\
\ge&V_i(\pi_{\mu_{i-1},\nu_{i-1}}\otimes\mu_{i|i-1}(d{\bf x}_{n_i}|{\bf y}_{n_{i-1}}),
f_{\nu_i}({\bf y}_{[n_{i-1}+1,n_i]}|{\bf y}_{n_{i-1}})d{\bf y}_{[n_{i-1}+1,n_i]};{\bf y}_{n_{i-1}}),
\notag
\end{align}
$f_{\nu_{i-1}}({\bf y}_{n_{i-1}})d{\bf y}_{n_{i-1}}{\rm -a.e.}$
(see (\ref{1.21.1221}) and (\ref{2.30})).
$\pi_{\mu_i, \nu_i}(d{\bf x}_{n_i}\hbox{ }d{\bf y}_{[n_{i-1}+1,n_i]}|{\bf y}_{n_{i-1}})$ is the unique minimizer of (\ref{4.4}),  $f_{\nu_{i-1}}({\bf y}_{n_{i-1}})d{\bf y}_{n_{i-1}}{\rm -a.e.}$,
provided it is finite, from Proposition \ref{pp2.1} (see (\ref{4.9}) and also Remark \ref{rk2.2.0}).
In particular, $\pi_{\mu_i, \nu_i}$ is the unique minimizer of $V_i(\mu_i, \nu_i;\pi_{\mu_{i-1},\nu_{i-1}})$, provided it is finite.
\end{proof}

\appendix
{\color{black}
\section{Proofs of Example \ref{ex2.1} and Lemma \ref{lm3.1}}

In this section, we give the proofs of Example \ref{ex2.1} and Lemma \ref{lm3.1}.

\begin{proof}[Proof of Example \ref{ex2.1}]
Let
\begin{align*}
\pi_i(d{\bf x}_{n_i}\hbox{ }d{\bf y}_{n_i}):=&\pi_{\mu_{i-1},\nu_{i-1}} \otimes \mu_{i|i-1}(d{\bf x}_{n_i}\hbox{ }d{\bf y}_{n_{i-1}}) f_{\nu_i}({\bf y}_{[n_{i-1}+1,n_i]}|{\bf y}_{n_{i-1}})d{\bf y}_{[n_{i-1}+1,n_i]}.
\end{align*}
Then $\pi_i\in \mathcal{A}(\mu_i,\nu_i;\pi_{\mu_{i-1},\nu_{i-1}})$, and 
\begin{align*}
V_i(\mu_i,\nu_i;\pi_{\mu_{i-1},\nu_{i-1}})
\le &H(\pi_i\parallel\pi_{0,i})\notag\\
=&\mathcal{S}(f_{\nu_i})-
H(f_{\nu_{i-1}}({\bf y}_{n_{i-1}})d{\bf y}_{n_{i-1}}\parallel g_{i-1}(1, {\bf y}_{n_{i-1}})d{\bf y}_{n_{i-1}})
\notag\\
&\qquad  
-\int_{\mathbb{R}^{n_{i-1}}}\{\log g_{i-1}(1, y)\}f_{\nu_{i-1}}(y)dy\notag\\
&\qquad+
\int_{\mathbb{R}^{n_{i-1}}\times\mathbb{R}^{n_{i-1}}}
\left\{\log p_{i-1}(x,y) \right\}\pi_{\mu_{i-1},\nu_{i-1}}(dx\hbox{ }dy)
\notag\\
&\qquad -\int_{\mathbb{R}^{n_i}\times\mathbb{R}^{n_i}}
\left\{\log p_{i}(x,y) \right\}\pi_i(dx\hbox{ }dy)<\infty,\notag
\end{align*}
where $g_{i-1}(1,y):=\prod_{j=1}^{n_{i-1}}g(1,y_j), y=(y_j)_{j=1}^{n_{i-1}}\in \mathbb{R}^{n_{i-1}}$, and integrating (\ref{2.15}) in ${\bf y}_{[n_{i-1}+1,n_i]}$
on $\mathbb{R}^{d_i}$,
$$p_{i-1}(x,y)\le 
C\sqrt{\pi C}^{d_i}, \quad x,y\in\mathbb{R}^{n_{i-1}}.$$
\end{proof}

\begin{proof}[Proof of Lemma \ref{lm3.1}]
We only prove (ii) since (i) can be proven similarly.
If 
\begin{equation}\label{3.4}
\int_{\mathbb{R}^{n_i}}\varphi_i(x)\mu_i(dx)=0,\tag{1}
\end{equation}
then
$\mathcal{I}_i(\varphi_i)\equiv 0$ and $Dom (\mathcal{I}_i(\varphi_i))=\mathbb{R}^{n_i}$.
In particular, $\mathcal{I}_i(\varphi_i)({\bf y}_{n_{i-1}},\cdot)\equiv 0$ and 
$Dom (\mathcal{I}_i(\varphi_i)({\bf y}_{n_{i-1}},\cdot))=\mathbb{R}^{d_i}$.
We consider the case where (\ref{3.4}) does not hold and 
$\mathcal{I}_i(\varphi_i)(y)>0, y\in \mathbb{R}^{n_i}$.
First, we prove that the function defined in the following is convex: for ${\bf y}_{n_{i-1}}\in \mathbb{R}^{n_{i-1}}$,
$$\mathbb{R}^{d_i}\ni y\mapsto \Phi_i({\bf y}_{n_{i-1}}, y;\varphi_i):=\psi_i(y)+\log \mathcal{I}_i(\varphi_i)({\bf y}_{n_{i-1}}, y).$$
Indeed, for $\lambda\in (0,1), y, z\in \mathbb{R}^{d_i}$, from (A0, iii), by H\"older's inequality,
\begin{align*}
&\Phi_i({\bf y}_{n_{i-1}}, \lambda y+(1-\lambda )z;\varphi_i)\\
=&\log \int_{\mathbb{R}^{n_i}}\exp \{\log p_i(x,({\bf y}_{n_{i-1}}, w))+\psi_i(w)\}|_{w=\lambda y+(1-\lambda )z}\varphi_i(x)\mu_i(dx)\notag\\
\le &\log \int_{\mathbb{R}^{n_i}}\exp \{\lambda (\log p_i(x,({\bf y}_{n_{i-1}},y))+\psi_i(y))\notag\\
&\qquad +(1-\lambda )(\log p_i(x,({\bf y}_{n_{i-1}},z))+\psi_i(z))\}\varphi_i(x)\mu_i(dx)\notag\\
\le &\lambda\Phi_i({\bf y}_{n_{i-1}}, y;\varphi_i) 
+(1-\lambda )\Phi_i({\bf y}_{n_{i-1}}, z;\varphi_i),\notag
\end{align*}
since $\varphi_i\ge 0$.
Since  $\mathbb{R}^{d_i}\ni y\mapsto \Phi_i({\bf y}_{n_{i-1}}, y;\varphi_i)$ is convex, 
$Dom (\Phi_i({\bf y}_{n_{i-1}}, \cdot;\varphi_i))$
is a convex set and $\Phi_i({\bf y}_{n_{i-1}}, \cdot;\varphi_i)$ is continuous in the interior of $Dom (\Phi_i({\bf y}_{n_{i-1}}, \cdot;\varphi_i))$ (see e.g., \cite{V1}, p. 52).
Since $\psi_i$ is continuous, $\mathcal{I}_i(\varphi_i)({\bf y}_{n_{i-1}},\cdot)$ 
 is continuous in the interior of the set $Dom (\Phi_i({\bf y}_{n_{i-1}}, \cdot;\varphi_i))=Dom (\mathcal{I}_i(\varphi_i)({\bf y}_{n_{i-1}},\cdot))$.
 \end{proof}
}



\end{document}